\documentclass{amsart}

\usepackage{amssymb}
\usepackage{mathtools} 
\usepackage{tikz} 
\usepackage{tikz-cd} 
\usepackage{enumitem} 
\usepackage{multirow} 
\usepackage{hyperref} 

\DeclareMathOperator{\codim}{codim}
\DeclareMathOperator{\coker}{coker}
\DeclareMathOperator{\Sing}{Sing}

\usepackage[normalem]{ulem}
\usepackage{algorithmic}
\usepackage[ruled]{algorithm}
\usepackage{caption}

\newtheorem{theorem}{Theorem}[section]
\newtheorem{proposition}[theorem]{Proposition}
\newtheorem{lemma}[theorem]{Lemma}
\newtheorem{corollary}[theorem]{Corollary}

\theoremstyle{definition}
\newtheorem{definition}[theorem]{Definition}
\newtheorem{example}[theorem]{Example}

\theoremstyle{remark}
\newtheorem{remark}[theorem]{Remark}
\newtheorem{notation}[theorem]{Notation}
\newtheorem*{warning}{Warning}

\numberwithin{equation}{section}

\title{An effective decomposition theorem for Schubert varieties}
\author[F. Cioffi, D. Franco and C. Sessa]{Francesca Cioffi \and Davide Franco \and Carmine Sessa}
\date{21 March 2022}
\thanks{\textit{Address}: Dipartimento di Matematica e Applicazioni "R. Caccioppoli", Università degli Studi di Napoli Federico II, Via Cintia, 80126 Napoli, Italy}
\thanks{\textit{emails}: cioffifr@unina.it \quad davide.franco@unina.it \quad carmine.sessa2@unina.it}
\thanks{All authors are members of the GNSAGA (INdAM, Italy).}

\begin{document}

\begin{abstract}
	Given a Schubert variety $\mathcal{S}$ contained in a Grassmannian $\mathbb{G}_{k}(\mathbb{C}^{l})$, we show how to obtain further information on the direct summands of the derived pushforward $R \pi_{*} \mathbb{Q}_{\tilde{\mathcal{S}}}$ given by the application of the decomposition theorem to a suitable resolution of singularities $\pi: \tilde{\mathcal{S}} \rightarrow \mathcal{S}$. As a by-product, Poincaré polynomial expressions are obtained along with an algorithm which computes the unknown terms in such expressions and which shows that the actual number of direct summands happens to be less than the number of supports of the decomposition.
	\\[2ex]
\end{abstract}

\maketitle

\section{Introduction}

Let $\mathcal{S} \subset \mathbb G_{k}(\mathbb{C}^{l})$ be a Schubert variety of a Grassmannian. In general, $\mathcal{S}$ is highly singular but it is well known that it admits a \emph{small resolution} $\sigma: \Sigma \rightarrow \mathcal{S}$ (see \cite{Zel1983}, \cite[\textsection 9.1]{BiLa2000}). Such a resolution is particularly useful since it determines the intersection cohomology complex of $\mathcal{S}$
\begin{equation*}
	R \sigma_{*} \mathbb{Q}_{\Sigma} [ \dim \mathcal{S} ] \cong IC_{\mathcal{S}}^{\bullet},
\end{equation*}
which is a constructible complex of $\mathbb{Q}$-vector sheaves on $\mathcal{S}$. In particular, for any $x \in \mathcal{S}$, there is a well defined polynomial ($\mathcal{H}^{\alpha}$ denotes the $\alpha$-th cohomology sheaf)
\begin{equation}\label{KLintro}
	b_{x} := \sum_{\alpha \in \mathbb{N}} \dim_{\mathbb{Q}} \mathcal{H}^{\alpha}(IC_{\mathcal{S}}^{\bullet} [- \dim \mathcal{S}])_{x} t^{\alpha},
\end{equation}
which is an example of Kazhdan-Lusztig polynomial. Indeed, the main advantage of the small resolution $\sigma: \Sigma \rightarrow \mathcal{S}$ lies in the fact that it enables the explicit determination of Kazhdan-Lusztig polynomials for the Schubert varieties of a Grassmannian, usually very hard to compute. On the other hand, $\sigma$ has several drawbacks; in fact, the construction of small resolutions given in \cite{Zel1983} is inductive, which makes the resolutions themselves not that explicit in the sense that their fibres, which are usually highly singular and reducible, are not easy to describe in general.

In this paper, our point of view is to consider another canonical resolution $\pi: \tilde{\mathcal{S}} \rightarrow \mathcal{S}$, usually not small, but presenting several nice features: it is completely explicit, its fibres are smooth and immediate to determine, their Poincar\'e polynomials are easy to compute and one can have thorough control of the strata where the fibres change. However, it has to be acknowledged that the computation of the Kazhdan-Lusztig polynomials requires an additional inductive formula, which is not needed while working with the small resolution $\sigma$.

Since $\pi$ is not usually small, the interplay between $R \pi_{*} \mathbb{Q}_{\tilde{\mathcal{S}}} [ \dim \mathcal{S} ]$ and the intersection cohomology of $\mathcal{S}$ is much more involved and governed by the \emph{decomposition theorem} (see Theorem~\ref{ThDec}, \cite{BeBeDe1982} and \cite[1.6.1]{dCaMi2009}). Indeed, one of the main consequences of such result is that the intersection cohomology complex of $\mathcal{S}$ is a direct summand of the complex $R \pi_{*} \mathbb{Q}_{\tilde{\mathcal{S}}} [ \dim \mathcal{S} ]$ in the derived category $D_{c}^{b}(\mathcal{S})$ of $\mathbb{Q}$-vector sheaves on $\mathcal{S}$. Specifically, by \cite[\textsection~1.5]{dCaMi2009}, the decomposition theorem applied to $\pi$ provides a non-canonical decomposition
\begin{equation}\label{decThmintr}
	R \pi_{*}\mathbb{Q}_{\tilde{\mathcal{S}}} [ \dim \mathcal{S} ] \cong \bigoplus_{\alpha \in \mathbb{Z}} \prescript{\mathfrak{p}}{}{\mathcal{H}}^{\alpha}(R \pi_{*}\mathbb{Q}_{\tilde{\mathcal{S}}} [ \dim \mathcal{S} ] ) \cong \bigoplus_{i \in \mathbb{Z}} \bigoplus_{j \in \mathbb{N}} IC^{\bullet}(L_{ij}) [-i],
\end{equation}
where $\prescript{\mathfrak{p}}{}{\mathcal{H}}^{\alpha}(R \pi_{*} \mathbb{Q}_{\tilde{\mathcal{S}}} [ \dim \mathcal{S} ] )$ denotes the \textit{perverse cohomology sheaves} \cite[\S 1.5]{dCaMi2009}. These sheaves are semisimple, i.e.~direct sums of intersection cohomology complexes of the semisimple local systems $L_{ij}$, each of which is supported on a suitable locally closed stratum of codimension $j$, usually called a \emph{support} of the decomposition. The summand supported in the general point is precisely the intersection cohomology of $\mathcal{S}$. The supports appearing in the splitting (\ref{decThmintr}) and the local systems $L_{ij}$ are, generally, rather mysterious objects when $j \geq 1$.

In literature one can find different approaches to the decomposition theorem (see \cite{BeBeDe1982, dCaMi2005, dCaMi2009, Sai1986, Wil2017}), which is a very general result but also rather implicit. On the other hand,  there are many special cases for which the decomposition theorem admits a simplified and explicit approach. One of these is the case of varieties with isolated singularities (see \cite{Nav1985,dGeFr2017OnTheTopology, dGeFr2020}). For instance, in the work \cite{dGeFr2017OnTheTopology}, a simplified approach to the decomposition theorem for varieties with isolated singularities is developed, in connection with the existence of a \textit{natural Gysin morphism}, as defined in \cite[Definition~2.3]{dGeFr2017OnTheExistence} (see also \cite{dGeFr2014} for other applications of the decomposition theorem to the Noether-Lefschetz Theory).

\textit{The main aim of this paper is to determine the summands  involved in (\ref{decThmintr}) and to provide an explicit description of the splitting}. 

A priori, all Schubert varieties contained in $\mathcal{S}$ might be expected to appear as supports in the splitting (\ref{decThmintr}); yet, we shall see that only some suitable Schubert subvarieties, that we call $\mathcal{S}$-varieties (see Definition~\ref{def:S-variety}), are allowed to appear.

The starting point of our analysis stems from the remark that \textit{the semisimple local systems involved in the decomposition are constant sheaves supported in the smooth part of the admissible strata} (cfr. Theorem~\ref{ThMain}). In other words, the decomposition (\ref{decThmintr}) takes the form
\begin{equation}\label{decimpr}
	R \pi_{*}\mathbb{Q}_{\tilde{\mathcal{S}}}  \cong \bigoplus_{h, k} IC_{\Delta_{h}}^{\bullet}[-k]^{\oplus n_{hk}},
\end{equation}
where $\Delta_{h}$ ranges among the $\mathcal{S}$-varieties and $n_{hk} \in \mathbb{N}_{0}$ denotes suitable multiplicities (see Remark~\ref{RemMultiplicities}). 

Some natural questions arise: how to compute the multiplicities $n_{hk}$? In particular, which strata actually appear in  (\ref{decimpr})? Is it possible to use (\ref{decimpr}) to compute Kazhdan-Lusztig polynomials?

Single condition Schubert varieties have been studied in \cite{dGeFr2019} and \cite{CiFrSe2021}. In this paper, we provide an algorithm to answer these questions for Schubert varieties with an arbitrary number of conditions. Specifically, we will prove that the Poincar\'e polynomials of the fibres of $\pi$ determine inductively both the generating functions of the multiplicities $n_{hk}$ and the Kazhdan-Lusztig polynomials and we deduce an algorithm, which we are going to name \textit{KaLu}, computing both of them.

Last but not the least, the just mentioned algorithm makes us observe the following fact. By analogy with \cite[Definition 4.2.3]{dCaMi2009}, we call an $\mathcal{S}$-variety $\Delta_{h} \subset \mathcal{S}$ \textit{$\pi$-relevant} if and only if
\begin{equation*}
	\dim S - \dim \Delta_{h} \leq 2 \dim \pi^{-1}(p),
\end{equation*}
where $p$ is a general point of $\Delta_{h}$. It is natural to ask whether \textit{all $\pi$-relevant varieties actually are supports in \eqref{decimpr}}. Thanks to some tests performed by means of our algorithm, we obtained a negative answer (see Section~\ref{SubsecRelevant}).

Let us conclude by summarizing the organization of the paper. In Section~\ref{SecSchVar} we settle notations once and for all and recall definitions and facts which are needed throughout the paper. In Section~\ref{SecMainTh} we prove our main result from which certain families of polynomial expressions are inferred. Section~\ref{SecKaLu} is devoted to the description of the algorithm which computes the unknown polynomials involved in the expressions mentioned above. Moreover, we also explain how we checked the correctness of the algorithm and how it finds out if a $\pi$-relevant variety actually gives a contribution in \eqref{decimpr}. Some ancillary files related to these tests and experiments are available at \url{http://wpage.unina.it/carmine.sessa2/KaLu}, along with an implementation of the algorithm \textit{KaLu} in CoCoA5 \cite{CoCoA5}. The last section is an appendix which consists of several examples concerning (some of) the properties of Schubert varieties that can be deduced by means of their Ferrer's diagrams (defined in Section~\ref{SubsecFerrersdiagrams}).

\section{Schubert varieties}\label{SecSchVar}

We shall begin by well known facts concerning Grassmannians and Schubert varieties. The reason why we decided to enlarge on them rests upon the convenience of setting notations once and for all and that we would like our paper to be comprehensible also to readers less familiar with the treated subjects.

Sections~\ref{SubsecDefinition} and \ref{SubsecFerrersdiagrams} are devoted to the description of Schubert varieties and their representation by means of Ferrer's diagrams. In Section~\ref{SubsecSubvarieties} we explain which are the subvarieties of a given Schubert variety $\mathcal{S}$ that are needed later on. We recommend looking at the examples available in Section~\ref{SecExamples}, which show how most properties of Schubert varieties are conveyed by their Ferrer's diagrams.

\subsection{Definition and remarks}\label{SubsecDefinition}

Throughout the paper, we shall work with $\mathbb{Q}$-coefficients cohomology and the middle (or self-dual) perversity $\mathfrak{p}$ (see \cite[\textsection 2.1]{BeBeDe1982}, \cite[\textsection 1.3]{GoMa1980} and \cite[p. 79]{GoMa1983}). To start with, let $k$ be a positive integer and let $H$ be a complex vector space. We shall denote by $\mathbb{G}_{k}(H) := \{ V \subseteq H: \dim V = k \}$ the \textit{Grassmannian of $k$-dimensional subspaces of $H$.}

Secondly, let $\psi: E \rightarrow B$ be a complex vector bundle, which can be also denoted by $E$ whenever $\psi$ and $B$ are clear in the context. We shall denote by $\mathcal{G}_{h}(E) \rightarrow B$ the \textit{Grassmannian $h$-plane bundle of $E$}, whose fibre at any $b \in B$ is $\mathcal{G}_{h}(E)_{b} = \mathbb{G}_{h}(\psi^{-1}(b))$.

Lastly, let $n$ and $l$ be positive integers. A \textit{partial flag (of length $n$)} in $\mathbb{C}^{l}$ is a finite sequence of vector subspaces $\mathcal{H}: H_{1} \subset \ldots \subset H_{n}$ with $H_{1} \neq 0$ and $H_{n} \subset \mathbb{C}^{l}$. Another flag $\mathcal{H}': H_{1}' \subset \ldots \subset H_{n'}'$ is said to be a \textit{subflag} of $\mathcal{H}$, and we shall write $\mathcal{H}' \subseteq \mathcal{H}$, if and only if for each $\alpha \in \{ 1, \ldots, n' \}$ there is $\beta \in \{ 1, \ldots, n \}$ such that $H_{\alpha}' = H_{\beta}$.

\begin{definition}
	Given a partial flag $\mathcal{F}: F_{j_{1}} \subset \dots \subset F_{j_{\omega}}$ of $\mathbb{C}^{l}$, where $\dim F_{j_{\alpha}} = j_{\alpha}$ for every $\alpha \in \{ 1, \dots, \omega \}$, and an $\omega$-tuple of non-negative integers $\mathcal{I} = (i_{1}, \dots, i_{\omega})$, the \textbf{Schubert variety associated to $\mathcal{F}$ and $\mathcal{I}$} is the subvariety of $\mathbb{G}_{k}(\mathbb{C}^{l})$ given by
	\begin{equation*}
		\mathcal{S} := \{ V \in \mathbb G_k(\mathbb{C}^{l}) \hspace{0.165em} : \hspace{0.165em} \dim( V \cap F_{j_{\alpha}}) \geq i_{\alpha}, \hspace{0.33em} \alpha \in \{ 1, \dots, \omega \} \}.
	\end{equation*}
\end{definition}

Let $\mathcal{S}$ be the Schubert variety associated to $(\mathcal{F}, \mathcal{I})$. From the definition, it immediately follows that $j_{1} < \ldots < j_{\omega} < l$. Notice that $\mathcal{S}$ is empty if and only if there is an index $\alpha$ such that $i_{\alpha} > \min \{ k, j_{\alpha} \}$ and that it is contained in a Grassmannian smaller than $\mathbb{G}_{k}(\mathbb{C}^{l})$ if and only if there is $\alpha$ such that $i_{\alpha} = \min \{ k, j_{\alpha} \}$. Hence, $\mathcal{S}$ is neither empty nor contained in a Grassmannian smaller than $\mathbb{G}_{k}(\mathbb{C}^{l})$ if we assume
\begin{equation*}
	0 < i_{\alpha} < k \quad \mbox{and} \quad i_{\alpha} < j_{\alpha} \hspace{0.33em} \forall \alpha \in \{ 1, \dots, \omega \}. 
\end{equation*} 
The correspondence between Schubert varieties and the pairs $(\mathcal{F}, \mathcal{I})$ is not bijective; in fact, some incidence conditions $\dim(V \cap F_{j_{\alpha}}) \geq i_{\alpha}$ might be superfluous (e.g., some of them may happen to be implied by the others). This situation does not occur if and only if, for every $\alpha$,
\begin{itemize}
	\item[(i)] $i_{\alpha} < i_{\alpha + 1}$; indeed, if $i_{\alpha} \geq i_{\alpha + 1}$ then the condition $\dim (V \cap F_{j_{\alpha + 1}}) \geq i_{\alpha + 1}$ is implied by the fact that $\dim (V \cap F_{j_{\alpha}})\geq i_{\alpha}$, being $F_{j_{\alpha}} \subset F_{j_{\alpha + 1}}$;
	\item[(ii)] $i_{\alpha + 1} - i_{\alpha} < j_{\alpha + 1} - j_{\alpha}$; indeed, if $i_{\alpha + 1} - i_{\alpha} \geq  j_{\alpha + 1} - j_{\alpha}$ then the condition $\dim (V \cap F_{j_{\alpha}}) \geq i_{\alpha}$ is implied by the fact that $\dim (V \cap F_{j_{\alpha + 1}}) \geq i_{\alpha + 1}$ because $\dim (V \cap F_{j_{\alpha + 1}})- \dim (V \cap F_{j_{\alpha}}) \leq j_{\alpha + 1} - j_{\alpha}$ by construction;
	\item[(iii)] $k + j_{\alpha} < l + i_{\alpha}$; indeed, if it is $k \geq l - j_{\alpha} + i_{\alpha}$ then the condition $\dim (V \cap F_{j_{\alpha}}) \geq i_{\alpha}$ is superfluous because $\dim (V \cap F_{j_{\alpha}}) = k + j_{\alpha} -\dim( V + F_{j_\alpha}) \geq l + i_{\alpha} - \dim(V + F_{j_\alpha}) \geq i_{\alpha}$.
\end{itemize}
Note that under the assumption (ii), condition (iii) is obtained by requiring $j_{\omega} - i_{\omega} < l - k$ only.

To sum up, although a Schubert variety $\mathcal{S}$ is given by a flag $\mathcal{F}$ and a vector $\mathcal{I} = (i_{1}, \ldots, i_{\omega})$, it is possible to get rid of the redundant conditions without changing the variety $\mathcal{S}$. In other words, we are allowed to remove some integers $i_{\alpha}$ from $\mathcal{I}$ and, consequently, the corresponding vector spaces $F_{j_{\alpha}}$ from $\mathcal{F}$ until we obtain the minimum information required to define $\mathcal{S}$. \textit{Let us see how to achieve that}.

Given $\mathcal{S}$ by means of a flag $\mathcal{F}: F_{j_{1}} \subset \ldots \subset F_{j_{\omega}}$ and of an $\omega$-tuple $\mathcal{I} = (i_{1}, \ldots, i_{\omega})$, set $\omega_{0} := \omega$, $\mathcal{I}_{0} = (i_{1}^{0}, \ldots, i_{\omega_{0}}^{0}) := \mathcal{I}$, $\mathcal{J}_{0} = (j_{1}^{0}, \ldots, j_{\omega_{0}}^{0}) := (j_{1}, \ldots, j_{\omega})$ and $\mathcal{F}_{0} := \mathcal{F}$.
\begin{description}[leftmargin = 5mm, before={\renewcommand\makelabel[1]{\bfseries ##1}}]
	\item[\textnormal{a)}]	For $\alpha = 1, \ldots, \omega_{0} - 1$,  if there is $\beta > \alpha$ such that $i_{\alpha}^{0} \geq i_{\beta}^{0}$, then delete the $\beta$-th condition; i.e.~set $\mathcal{I}_{0} := (i_{1}^{0}, \ldots, \hat{i}_{\beta}^{0}, \ldots, i_{\omega_{0}}^{0})$, $\mathcal{J}_{0} := (j_{1}^{0}, \ldots, \hat{j}_{\beta}^{0}, \ldots, j_{\omega_{0}}^{0})$, $\mathcal{F}_{0}: F_{j_{1}^{0}} \subset \ldots \subset \hat{F}_{j_{\beta}^{0}} \subset \ldots \subset F_{j_{\omega_{0}}^{0}}$, where the symbol $\hat{ }$ is used to indicate the term to delete, and $\omega_{0} := \omega_{0} - 1$;
	\item[\textnormal{b)}] for $\alpha = \omega_{0}, \ldots, 2$, if there is $\beta < \alpha$ such that $i_{\alpha}^{0} \geq j_{\alpha}^{0} - j_{\beta}^{0} + i_{\beta}^{0}$, then delete the $\beta$-th condition as explained in item a);
	\item[\textnormal{c)}] if $j_{\omega_{0}}^{0} - i_{\omega_{0}}^{0} \geq l - k$, then delete the last conditions as explained in item a) and go on until there is $\alpha < \omega$ such that $j_{\alpha} - i_{\alpha} < l - k$.
\end{description}
In case we do not want $\mathcal{S}$ to be contained in a Grassmannian smaller than $\mathbb{G}_{k}(\mathbb{C}^{l})$, we perform the following further control.
\begin{description}[leftmargin = 5mm, before={\renewcommand\makelabel[1]{\bfseries ##1}}]
	\item[\textnormal{d)}] If $i_{\omega_{0}}^{0} = k$, delete the last condition as shown in item a).
\end{description}

\begin{definition}
	Let $\mathcal{S}$ be the Schubert variety given by $(\mathcal{F}, \mathcal{I})$. The flag $\mathcal{F}$, the $\omega$-tuple $\mathcal{I}$ and the pair $(\mathcal{F}, \mathcal{I})$ are said to be \textbf{essential} if and only if they are the minimum information needed to define $\mathcal{S}$. Equivalently,
	\begin{equation*}
		0 < i_{1} < \ldots < i_{\omega} \leq k < l + i_{\omega} - j_{\omega}, \qquad i_{\alpha} < j_{\alpha} \hspace{0.33em} \forall \alpha
	\end{equation*}
and
	\begin{equation*}
		i_{\alpha + 1} - i_{\alpha} < j_{\alpha + 1} - j_{\alpha} \hspace{0.33em} \forall \alpha < \omega.
	\end{equation*}
	In particular, $\mathcal{S}$ is called either a \textbf{special} or a \textbf{single condition} Schubert variety if $\omega = 1$ and $i_{\omega} < k$.
\end{definition}

Notice that the property of being essential implies $\omega \leq k$, otherwise $i_{\omega} > k$, against the above conditions. Moreover, \textit{the correspondence between Schubert varieties and essential pairs $(\mathcal{F}, \mathcal{I})$ is bijective}. 

Instead, if $\mathcal{S}$ is a Schubert variety described by a flag $\mathcal{F}$ which is not essential and we do not want to omit the redundant conditions, there are different vectors $\mathcal{I}$ which describe $\mathcal{S}$ with respect to $\mathcal{F}$. In fact, if a condition $\dim (V \cap F_{j_{\alpha}}) \geq i_{\alpha}$ is superfluous, i.e.~either $i_{\alpha - 1} \geq i_{\alpha}$ or $i_{\alpha + 1} - i_{\alpha} \geq j_{\alpha + 1} - j_{\alpha}$ or $i_{\alpha} \leq k - l + j_{\alpha}$, then $\mathcal{S}$ does not change if we require either $\dim (V \cap F_{j_{\alpha}}) \geq i_{\alpha - 1} - \beta$ or $\dim (V \cap F_{j_{\alpha}}) \geq i_{\alpha + 1} - j_{\alpha + 1} + j_{\alpha} - \beta$ or $i_{\alpha} \leq k - l + j_{\alpha} - \beta$, respectively, whatever $\beta \geq 0$ is.

Hence, given a flag $\mathcal{F}$, there is a bijection between Schubert varieties and pairs $(\mathcal{F}, \mathcal{I})$ such that
\begin{equation}\label{eq:weak conditions}
	\begin{split}
		&0 \leq i_{1} \leq \ldots \leq i_{\omega} \leq k \leq l + i_{\omega} - j_{\omega}, \qquad i_{\alpha} \leq j_{\alpha} \hspace{0.33em} \forall \alpha\\
		&\mbox{and} \quad i_{\alpha + 1} - i_{\alpha} \leq j_{\alpha + 1} - j_{\alpha} \hspace{0.33em} \forall \alpha < \omega.
	\end{split}
\end{equation}

\subsection{Ferrer's diagrams}\label{SubsecFerrersdiagrams}

\begin{definition}
	Let $\lambda = (\lambda_{1}, \ldots, \lambda_{k})$ be a decreasing sequence of $k$ non-negative integers. The \textbf{Ferrer's diagram} of $\lambda$ is the diagram obtained by piling up $k$ rows of length $\lambda_{1}, \ldots, \lambda_{k}$, from top to bottom, so that their left edges are aligned.
\end{definition}

\begin{example}
	The Ferrer's diagram of $\lambda = (6, 6, 5, 4, 3)$ is
	
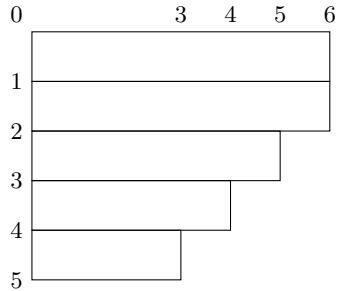
\begin{equation*}
		\begin{tikzpicture}[scale = 0.66]
			\draw (0, 0) node[anchor = south east] {\small{0}};
			\draw (0, -1) node[anchor = east] {\small{1}};
			\draw (0, -2) node[anchor = east] {\small{2}};
			\draw (0, -3) node[anchor = east] {\small{3}};
			\draw (0, -4) node[anchor = east] {\small{4}};
			\draw (0, -5) node[anchor = east] {\small{5}};
			\draw (3, 0) node[anchor = south] {\small{3}};
			\draw (4, 0) node[anchor = south] {\small{4}};
			\draw (5, 0) node[anchor = south] {\small{5}};
			\draw (6, 0) node[anchor = south] {\small{6}};
			\draw (0, -1) rectangle (6, 0);
			\draw (0, -2) rectangle (6, -1);
			\draw (0, -3) rectangle (5, -2);
			\draw (0, -4) rectangle (4, -3);
			\draw (0, -5) rectangle (3, -4);
		\end{tikzpicture}
	\end{equation*}
\end{example}

Given a nonempty Schubert variety $\mathcal{S}$, it is possible to associate to it the sequence of integers $\lambda^{\mathcal{S}} = (\lambda_{\alpha}^{\mathcal{S}})_{\alpha = 1, \ldots, k}$ defined as follows:
\begin{equation*}
	\lambda_{\alpha}^{\mathcal{S}} =
	\begin{cases*}
		l - k - j_{1} + i_{1} &\mbox{if $\alpha \in \{ 1, \ldots, i_{1} \}$}\\
		l - k - j_{2} + i_{2} &\mbox{if $\alpha \in \{ i_{1} + 1, \ldots, i_{2} \}$}\\
		\dots\\
		l - k - j_{\omega} + i_{\omega} &\mbox{if $\alpha \in \{ i_{\omega - 1} + 1, \ldots, i_{\omega} \}$}\\
		0 &\mbox{if $\alpha \in \{ i_{\omega} + 1, \ldots, k \}$}.
	\end{cases*}
\end{equation*}
It is worth pointing out that \textit{$\lambda^{\mathcal{S}}$ is independent of the choice of the flag $\mathcal{F}$ and the $\omega$-tuple $\mathcal{I}$.} Moreover, from the definition of Schubert varieties, it follows that the sequence \textit{$\lambda^{\mathcal{S}}$ is decreasing, with each entry non-negative and strictly lower than $l - k$}. Therefore, we can consider the Ferrer's diagram of $\lambda^{\mathcal{S}}$, which shall be called the \textit{Ferrer's diagram of $\mathcal{S}$}. When $(\mathcal F,\mathcal I)$ is essential, $\lambda^{\mathcal{S}}$ contains exactly $\omega$ different integers with their repetitions, if any.

Several properties of Schubert varieties are conveyed by their Ferrer's diagrams, as we will see throughout the paper. At the moment, let us just observe that \textit{the (complex) codimension with respect to $\mathbb{G}_{k}(\mathbb{C}^{l})$ of the Schubert variety associated to the sequence $\lambda^{\mathcal{S}} = (\lambda_{\alpha}^{\mathcal{S}})_{\alpha = 1, \ldots, k}$ equals the area of its Ferrer's diagram} \cite[pp. 194-196]{GrHa1994}.

\begin{warning}
	In the rest of the paper, we are going to assume that $k \leq j_{1}$. Let us explain the reason why such an assumption is sensible. First of all, notice that, whenever $i_{1}, \ldots, i_{\omega}$ and $k$ have been chosen, there are only a finite number of cases in which $k > j_{1}$. Secondly, a Schubert variety $\mathcal{S}$ in a Grassmannian $\mathbb{G}_{k}(\mathbb{C}^{l})$ can be thought of as the intersection of $\mathbb{G}_{k}(\mathbb{C}^{l})$ with the Schubert variety $\mathcal{S}'$ in the Grassmannian $\mathbb{G}_{k}(\mathbb{C}^{l + 1})$ which is represented by the same Ferrer's diagram of $\mathcal{S}$ (this means that $j_{\alpha}' = j_{\alpha} + 1$ for each $\alpha \in \{ 1, \ldots, \omega \}$). Therefore, the properties of $\mathcal{S}$ can be deduced from the ones of $\mathcal{S}'$. Lastly, another difference between $\mathcal{S}$ and $\mathcal{S}'$ is that the number of supports involved in the Decomposition Theorem is maximum when $k = j_{1}$, it does not change if $k < j_{1}$ and it lowers as $k > j_{1}$ increases (see Example~\ref{Exakj}).
\end{warning}

\subsection{Families of subvarieties}\label{SubsecSubvarieties}

\textit{Throughout the paper, we let $\mathcal{S}$ be a non-empty Schubert variety associated to the essential flag $\mathcal{F}: F_{j_{1}} \subset \ldots \subset F_{j_{\omega}}$ and $\omega$-tuple $\mathcal{I} = (i_{1}, \ldots, i_{\omega})$ with $i_{\omega} < k$}.

\begin{definition}\label{def:S-variety}
	An $\mathcal{S}$-\textbf{variety} is a non-empty Schubert subvariety of $\mathcal S$ associated to a subflag $\mathcal{F}_{p} : F_{j_{1}^{p}} \subset \ldots \subset F_{j_{\omega_{p}}^{p}}$ of the essential flag $\mathcal{F}$. 
	
	Equivalently, letting $\mathcal{I}_{p} := (i_1^p,\dots,i_{\omega_p}^p) + p$, where $i_\alpha^p$ is the component of $\mathcal I$ corresponding to $F_{j_\alpha^p}$ and $p := (p_{1}, \ldots, p_{\omega_{p}})$ is an $\omega_{p}$-tuple of non-negative integers, then 
	\begin{equation*}
		\Delta_{p} = \{ V \in \mathbb{G}_{k}(\mathbb{C}^{l}) \hspace{0.165em} : \hspace{0.165em} \dim (V \cap F_{j_{\alpha}^{p}}) \geq i_{\alpha}^{p} + p_{\alpha}, \hspace{0.33em} \alpha = 1, \ldots, \omega_{p} \}
	\end{equation*}
	is the $\mathcal{S}$-variety associated to $\mathcal{F}_{p}$ and $\mathcal{I}_{p}$ if and only if the pair $(\mathcal F_p,\mathcal I_p)$ satisfies conditions \eqref{eq:weak conditions}. The vector $p$ is said to be \textbf{essential} if $(\mathcal F_p,\mathcal I_p)$ is essential.
\end{definition}

Notice that \textit{$\mathcal{S}$ is the $\mathcal{S}$-variety given by $p = (0, \ldots, 0)$}.

\begin{remark} With the notation of Definition~\ref{def:S-variety}, observe that the pair $(\mathcal F_p,\mathcal I_p)$ satisfies conditions \eqref{eq:weak conditions} if and only if
	\begin{equation*}
		\begin{cases*}
			0 \leq p_{1} \leq k - i_{1}^{p}\\
			p_{\omega_{p}} \geq k - l + j_{\omega_{p}}^{p} - i_{\omega_{p}}^{p}\\
			M_{\alpha + 1} \leq p_{\alpha + 1} \leq N_{\alpha + 1} \hspace{0.33em} \forall \alpha = 1, \ldots, \omega_{p} - 1
		\end{cases*}
	\end{equation*}
	with
	\begin{align*}
		M_{\alpha + 1} &= \max \{ 0, i_{\alpha}^{p} + p_{\alpha} - i_{\alpha + 1}^{p} \}\\
		N_{\alpha + 1} &= \min \{ j_{\alpha + 1}^{p} - j_{\alpha}^{p} + i_{\alpha}^{p} + p_{\alpha} - i_{\alpha + 1}^{p}, k - i_{\alpha + 1}^{p} \}.
	\end{align*} 
\end{remark}

In Section~\ref{SubsecFerrersdiagrams} we described a way to represent $\mathcal{S}$ by means of its Ferrer's diagram. Needless to say, we can depict all $\mathcal{S}$-varieties $\Delta_{p}$ in the same way and \textit{we denote their associated sequences by} $\lambda^{p}$.

Let $\Delta_{p}$ be an $\mathcal{S}$-variety. As observed in Section~\ref{SubsecDefinition}, not all the conditions of $\Delta_p$ are supposed to be indispensable and we can find the necessary ones by means of the procedure explained there. 

\begin{notation}
	Let $\Delta_p$ be an $\mathcal S$-variety. If $p = (p_{1}, \ldots, p_{\omega_p})$ is not essential, \textit{we will denote by $(\mathcal{F}_{\bar{p}},\mathcal{I}_{\bar{p}})$ the essential pair to which $\Delta_{p}$ is associated}, being $\bar{p}$ the vector obtained by $p$ by deleting the components corresponding to redundant conditions. When we think of $\Delta_p$ as the Schubert variety associated to its essential pair $(\mathcal{F}_{\bar{p}}, \mathcal{I}_{\bar{p}})$, we will denote it by $\Delta_{\bar{p}}$:
	\begin{equation*}
		\Delta_{p} = \Delta_{\bar p} = \{ V \in \mathbb{G}_{k}(\mathbb{C}^{l}) \hspace{0.165em} : \hspace{0.165em} \dim (V \cap F_{j_{\alpha}^{\bar{p}}}) \geq i_{\alpha}^{\bar{p}} + \bar{p}_{\alpha}, \hspace{0.33em} \alpha = 1, \ldots, \omega_{\bar{p}} \}.
	\end{equation*}
\end{notation}

Notice that \textit{$p$ is essential if and only if $p = \bar{p}$}.

If $\Delta_{p}$ and $\Delta_{q}$ are $\mathcal{S}$-varieties, we say that $\Delta_q$ is a \textbf{$\Delta_{p}$-variety} if it has the properties written in Definition~\ref{def:S-variety} with $\mathcal{S}$ and $(\mathcal{F}, \mathcal{I})$ replaced by $\Delta_{p}$ and its essential pair $(\mathcal{F}_{\bar{p}}, \mathcal{I}_{\bar{p}})$. In this case, $q$ is said to be $p$\textbf{-admissible}; when $\Delta_{p} = \mathcal{S}$, we shall simply say that $q$ is \textbf{admissible}.

Let us now study \textit{the inclusion relation on the family of $\mathcal{S}$-varieties}. Let $\Delta_{p}$ and $\Delta_{q}$ be two $\mathcal{S}$-varieties associated to $(\mathcal{F}_{p}, \mathcal{I}_{p})$ and $(\mathcal{F}_{q}, \mathcal{I}_{q})$, respectively. If $\mathcal{F}_{p} = \mathcal{F}_{q} = \mathcal{F}$, it is straightforward to see that $\Delta_{q} \subseteq \Delta_{p}$ if and only if $q_{\alpha} \geq p_{\alpha}$ for any $\alpha \in \{ 1, \ldots, \omega \}$. In the general case, we can change the pairs $(\mathcal{F}_{p}, \mathcal{I}_{p})$ and $(\mathcal{F}_{q}, \mathcal{I}_{q})$ by adding redundant conditions so as to have $\mathcal{F}_{p} = \mathcal{F}_{q} = \mathcal{F}$ again (see Example~\ref{ExaNotComp}).

\begin{notation}\label{not:ordinamento}
	Given two $\mathcal{S}$-varieties $\Delta_{p}$, $\Delta_{q}$, we set $p \leq q \Leftrightarrow \Delta_{q} \subseteq \Delta_{p}$. If $\Delta_{q} \subseteq \Delta_{p}$, we also set $\vert q - p \vert := \sum q_{\alpha} - p_{\alpha}$ and call it the \textit{distance} between $p$ and $q$.
\end{notation}

In terms of Ferrer's diagrams, we have that $\Delta_{q} \subseteq \Delta_{p}$ \textit{if and only if the Ferrer's diagram of} $\Delta_{p}$ \textit{is contained in the one of} $\Delta_{q}$ (see \cite[Proposition 3.2.3 (4)]{Man2001}). When $\mathcal{S}$ is a special Schubert variety, $p$ and $q$ are integers and, as such, comparable. Consequently, \textit{the set of all $\mathcal{S}$-varieties is totally ordered by inclusion}. On the contrary, when $\omega > 1$, $\Delta_{p}$ and $\Delta_{q}$ are unlikely to be comparable with respect to the inclusion relation (see Examples~\ref{ExaNotComp} and \ref{ExaComp}).

\begin{remark}
	If $\Delta_{p}$ is an $\mathcal{S}$-variety, then the families of $\Delta_{p}$-varieties and $\mathcal{S}$-varieties contained in $\Delta_{p}$ do not coincide, unless $\mathcal{F}_{\bar{p}} = \mathcal{F}$. Indeed, the notion of $\Delta_{p}$-variety is stronger (see Example~\ref{ExaComp}).
\end{remark}

Later, for any chosen $\mathcal{S}$-variety $\Delta_{p}$, we will be interested in the $\Delta_{p}$-varieties. However, several results (see Section~\ref{SecMainTh}) provide useful information on $\mathcal{S}$-varieties $\Delta_{q} \subset \Delta_{p}$ which are not $\Delta_{p}$-varieties if we make the following association.

\begin{notation}
	Let $\Delta_{p}$ and $\Delta_{q}$ be $\mathcal{S}$-varieties associated to flags $\mathcal{F}_{p} \subseteq \mathcal{F}_{q}$. We set $q^{p} = (q_{1}^{p}, \ldots, q_{\omega_{p}}^{p})$, where $q_{\alpha}^{p}$ is the component of $q$ in the position of the $\alpha$-th necessary condition of $\Delta_{p}$ (in other words, $q^{p}$ is obtained by $q$ by considering the same entries as $p$ in the hypothesis that $p = \bar{p}$).
\end{notation}

Observe that, by construction, $\Delta_{q^{p}}$ is a $\Delta_{p}$-variety; in particular $\Delta_{\bar{q}} = \Delta_{q^{p}}$ if and only if $\Delta_{q}$ is a $\Delta_{p}$-variety (see Example~\ref{ExaAdapt}). Furthermore, if $\Delta_{q}$ is not a $\Delta_{p}$-variety and $(\mathcal F_{q^{p}}, \mathcal I_{q^{p}})$ is replaced by the pair whose flag is $\mathcal F_{p}$, then we have $\vert q^{p} - p \vert < \vert q - p \vert$, according to Notation~\ref{not:ordinamento}.

\section{Schubert varieties and Decomposition theorem}\label{SecMainTh}

Here, we are going to define a class of resolution of singularities $\pi_{p}: \tilde{\Delta}_{p} \rightarrow \Delta_{p}$ (see Section~\ref{SubsecCommSq}). One of the main reasons why we chose this particular family is that we can always control the fibres of these maps, which will be fundamental for our purposes. In Section~\ref{SubsecDecTh}, we apply decomposition theorem to them so as to obtain information on the direct summands appearing in \eqref{decThmintr} (see Theorem~\ref{ThMain}) and, in Section~\ref{SubsecPolExpr}, certain classes of polynomial expressions.

\subsection{A family of resolution of singularities}\label{SubsecCommSq}

Let $H_{1} \subset \ldots \subset H_{n}$ be complex vector spaces and let $k_{1}, \ldots, k_{n}$ be positive integers such that $k_{\alpha} < \dim H_{\alpha}$ for any $\alpha = 1, \ldots, n$. Put
\begin{align*}
	\mathbb{F}(k_{1}, \ldots, k_{n}; H_{1}, \ldots, H_{n}) :=
	\begin{Bmatrix*}
		(K_{1}, \ldots, K_{n})\\
		\in \mathbb{G}_{k_{1}}(H_{1}) \times \ldots \times \mathbb{G}_{k_{n}}(H_{n})\\
		\mbox{s.t.} \quad \quad K_{1} \subset \ldots \subset K_{n}
	\end{Bmatrix*}.
\end{align*}

\begin{proposition}\label{PropFlagSmooth}
	$\mathbb{F}(k_{1}, \ldots, k_{n}; H_{1}, \ldots, H_{n})$ is smooth.
\end{proposition}

\begin{proof}
	If $n = 1$, $\mathbb{F}(k_{1}; H_{1}) = \mathbb{G}_{k_{1}}(H_{1})$ is smooth.
	
	Let $n \geq 2$.
	There is a chain of projections
	\begin{equation*}
		\begin{tikzcd}[column sep = small, row sep = tiny]
			\mathbb{F}(k_{1}, \ldots, k_{n}; H_{1}, \ldots, H_{n}) \arrow[r] \arrow[d, phantom, ""{coordinate, name=Z}] & \mathbb{F}(k_{1}, \ldots, k_{n - 1}; H_{1}, \ldots, H_{n - 1}) \arrow[r] & \dots \arrow[dll, rounded corners,
			to path={ -- ([xshift=2ex]\tikztostart.east)
				|- (Z) [near end]\tikztonodes
				-| ([xshift=-2ex]\tikztotarget.west)
				-- (\tikztotarget)}]\\
			\mathbb{F}(k_{1}, k_{2}; H_{1}, H_{2}) \arrow[r] & \mathbb{G}_{k_{1}}(H_{1}) &
		\end{tikzcd}
	\end{equation*}
	and each $\mathbb{F}(k_{1}, \ldots, k_{\alpha}; H_{1}, \ldots, H_{\alpha})$ is the Grassmannian bundle of a vector bundle over $\mathbb{F}(k_{1}, \ldots, k_{\alpha - 1}; H_{1}, \ldots, H_{\alpha - 1})$. In fact, for any $2 \leq \alpha \leq n$, there is an exact sequence of vector bundles
	\begin{equation*}
		\begin{tikzcd}[column sep = small]
			0 \arrow[r] & S_{\mathbb{G}_{k_{\alpha - 1}}(H_{\alpha - 1})} \arrow[r] \arrow[dr] & H_{\alpha} \arrow[r] \arrow[d] & Q_{\alpha - 1} \arrow[r] \arrow[dl] & 0\\
			& & \mathbb{G}_{k_{\alpha - 1}}(H_{\alpha - 1}) & &
		\end{tikzcd}
	\end{equation*}
	where $S_{\mathbb{G}_{k_{\alpha - 1}}(H_{\alpha - 1})}$ and $H_{\alpha}$ are, respectively, the tautological and trivial bundle over $\mathbb{G}_{k_{\alpha - 1}}(H_{\alpha - 1})$, while $Q_{\alpha - 1} = \coker (S_{\mathbb{G}_{k_{\alpha - 1}}(H_{\alpha - 1})} \rightarrow H_{\alpha})$.

	If we denote by $\psi_{\alpha - 1}: \mathbb{F}(k_{1}, \ldots, k_{\alpha - 1}; H_{1}, \ldots, H_{\alpha - 1}) \rightarrow \mathbb{G}_{k_{\alpha - 1}}(H_{\alpha - 1})$ the projection map and by $\psi_{\alpha - 1}^{*}$ its pullback, we have
	\begin{equation*}
		\mathbb{F}(k_{1}, \ldots, k_{\alpha}; H_{1}, \ldots, H_{\alpha}) \cong \mathcal{G}_{k_{\alpha} - k_{\alpha - 1}}(\psi_{\alpha - 1}^{*} Q_{\alpha - 1}).\qedhere
	\end{equation*}
\end{proof}

Let us go back to Schubert varieties. Given an $\mathcal{S}$-variety $\Delta_{p}$, put
\begin{equation*}
	\Delta_{p}^{0} = \{ V \in \mathbb{G}_{k}(\mathbb{C}^{l}) \hspace{0.165em} : \hspace{0.165em} \dim (V \cap F_{j_{\alpha}^{\bar{p}}}) = i_{\alpha}^{\bar{p}} + \bar{p}_{\alpha}, \hspace{0.33em} \alpha = 1, \ldots, \omega_{\bar{p}} \}.
\end{equation*}
$\Delta_{p}^{0}$ is a dense subset of $\Delta_{p}$; in fact, if we think of $\mathcal{F}_{\bar{p}}$ as the subflag of some complete flag $\mathcal{F}_{com}: F_{1} \subset \ldots \subset F_{l - 1}$, it is possible to prove that it contains the set
\begin{equation*}
	\Omega_{p} := \begin{Bmatrix*}
		V \in \mathbb{G}_{k}(\mathbb{C}^{l}) \hspace{0.165em} : \hspace{0.165em} \dim V \cap F_{\beta} = \alpha\\
		\mbox{if} \quad l - k + \alpha - \lambda_{\alpha}^{p} \leq \beta \leq l - k + \alpha - \lambda_{\alpha + 1}^{p}
	\end{Bmatrix*}
	,
\end{equation*}
which is a dense subset of $\Delta_{p}$ \cite[Proposition 3.2.3]{Man2001}. Furthermore,

\begin{proposition}
	$\Delta_{p}^{0}$ is the smooth locus of $\Delta_{p}$. In particular, it is a locally closed subset of $\mathcal{S}$.
\end{proposition}
\begin{proof}
	We want to prove that $\Delta_{p}^{0} = \Delta_{p} \backslash \Sing \Delta_{p}$.
	
	W.l.o.g.~we can assume that $p = \bar{p}$. The singular locus of $\Delta_{p}$ coincides with the union of all $\Delta_{p}$-varieties whose distance from $\Delta_{p}$ is~$1$ \cite[Example 3.4.3, Theorem 3.4.4]{Man2001}. As a consequence, the smooth locus of $\Delta_{p}$ is
	\begin{align*}
		&\Delta_{p} \backslash (\Delta_{(1, 0, \ldots, 0)} \cup \ldots \cup \Delta_{(0, \ldots, 0, 1)})\\
		&\qquad = \Delta_{p} \backslash \Delta_{(1, 0, \ldots, 0)} \cap \ldots \cap \Delta_{p} \backslash \Delta_{(0, \ldots, 0, 1)}\\
		&\qquad = \{ V \in \Delta_{p} \hspace{0.165em} : \hspace{0.165em} \dim(V \cap F_{j_{1}^{p}}) = i_{1}^{p} + p_{1} \} \cap \ldots\\
		&\qquad \qquad \cap \{ V \in \Delta_{p} \hspace{0.165em} : \hspace{0.165em} \dim(V \cap F_{j_{\omega_{p}}^{p}}) = i_{\omega_{p}}^{p} + p_{\omega_{p}} \}\\
		&\qquad = \{ V \in \Delta_{p} \hspace{0.165em} : \hspace{0.165em} \dim(V \cap F_{j_{\alpha}^{p}}) = i_{\alpha}^{p} + p_{\alpha}, \hspace{0.33em} \alpha = 1, \ldots, \omega_{p} \}\\
		&\qquad = \Delta_{p}^{0}. \qedhere
	\end{align*}
\end{proof}

Now, set
\begin{align*}
	\tilde{\Delta}_{p} &:= \mathbb{F}(i_{1}^{\bar{p}} + \bar{p}_{1}, \ldots, i_{\omega_{\bar{p}}}^{\bar{p}} + \bar{p}_{\omega_{\bar{p}}}, k; F_{j_{1}^{\bar{p}}}, \ldots, F_{j_{\omega_{\bar{p}}}^{\bar{p}}}, \mathbb{C}^{l})\\
	&= \begin{Bmatrix*}
		(Z_{1}, \ldots, Z_{\omega_{\bar{p}}}, V)\\
		\in \mathbb{G}_{i_{1}^{\bar{p}} + \bar{p}_{1}}(F_{j_{1}^{\bar{p}}}) \times \ldots \times \mathbb{G}_{i_{\omega_{\bar{p}}}^{\bar{p}} + \bar{p}_{\omega_{\bar{p}}}}(F_{j_{\omega_{\bar{p}}}^{\bar{p}}}) \times \mathbb{G}_{k}(\mathbb{C}^{l})\\
		\mbox{s.t.} \quad Z_{1} \subset \ldots \subset Z_{\omega_{\bar{p}}} \subset V
	\end{Bmatrix*}
	.
\end{align*}

\begin{corollary}\label{CorResolPI}
	$\tilde{\Delta}_{p}$ is smooth and the projection
	\begin{equation*}
		\pi_{p}: (Z_{1}, \ldots, Z_{\omega_{\bar{p}}}, V) \in \tilde{\Delta}_{p} \mapsto V \in \Delta_{p}.
	\end{equation*}
	is a resolution of singularities.
\end{corollary}
\begin{proof}
	Smoothness is a consequence of Lemma~\ref{PropFlagSmooth}.

	If $V \in \Delta_{p}^{0}$, then $\dim (V \cap F_{j_{\alpha}^{\bar{p}}}) = i_{\alpha}^{\bar{p}} + \bar{p}_{\alpha}$ for all $\alpha = 1, \ldots, \omega_{p}$ and, consequently,
	\begin{equation*}
		\pi_{p}^{-1}(V) \cong \{ (V \cap F_{j_{1}^{\bar{p}}}, \ldots, V \cap F_{j_{\omega_{\bar{p}}}^{\bar{p}}}) \}
	\end{equation*}
	gives the inverse map of $\pi_{p}$ on the open set $\Delta_{p}^{0}$.
\end{proof}

Let $\Delta_{q} \subset \Delta_{p}$ be two $\mathcal{S}$-varieties. Set
\begin{align*}
	\Delta_{pq}^{0} := \pi_{p}^{-1}(\Delta_{q}^{0}) = \begin{Bmatrix*}
		(Z_{1}, \ldots, Z_{\omega_{\bar{p}}}, V) \in \tilde{\Delta}_{p}\\
		\mbox{s.t.} \quad \dim (V \cap F_{j_{\alpha}^{\bar{p}}}) = i_{\alpha}^{\bar{p}} + q_{\alpha}^{p}, \hspace{0.33em} \alpha = 1, \ldots, \omega_{\bar{p}}
	\end{Bmatrix*}
	.
\end{align*}
The restriction of $\pi_{p}$
\begin{equation*}
	\rho_{pq}: (Z_{1}, \ldots, Z_{\omega_{\bar{p}}}, V) \in \Delta_{pq}^{0} \mapsto V \in \Delta_{q}^{0}
\end{equation*}
is a smooth and proper fibration with fibres
\begin{align*}
	F_{pq} := \rho_{pq}^{-1}(V) &\cong
	\begin{Bmatrix*}
		(Z_{1}, \ldots, Z_{\omega_{\bar{p}}})\\
		\in \mathbb{G}_{i_{1}^{\bar{p}} + \bar{p}_{1}}(V \cap F_{j_{1}^{\bar{p}}}) \times \ldots \times \mathbb{G}_{i_{\omega_{\bar{p}}}^{\bar{p}} + \bar{p}_{\omega_{\bar{p}}}}(V \cap F_{j_{\omega_{\bar{p}}}^{\bar{p}}})\\
		\mbox{s.t.} \quad \dim (V \cap F_{j_{\alpha}^{\bar{p}}}) = i_{\alpha}^{\bar{p}} + q_{\alpha}^{p}, \hspace{2mm} \alpha = 1, \ldots, \omega_{\bar{p}}
	\end{Bmatrix*}\\
	&\cong \mathbb{F}(i_{1}^{\bar{p}} + \bar{p}_{1}, \ldots, i_{\omega_{\bar{p}}}^{\bar{p}} + \bar{p}_{\omega_{\bar{p}}}; \mathbb{C}^{i_{1}^{\bar{p}} + q_{1}^{p}}, \ldots, \mathbb{C}^{i_{\omega_{\bar{p}}}^{\bar{p}} + q_{\omega_{\bar{p}}}^{\bar{p}}}),
\end{align*}
whose dimensions are
\begin{equation*}
	\dim F_{pq} = (q_{1}^{p} - \bar{p}_{1})(i_{1}^{\bar{p}} + \bar{p}_{1}) + \sum_{\alpha = 2}^{\omega_{\bar{p}}} (q_{\alpha}^{p} - \bar{p}_{\alpha})(i_{\alpha}^{\bar{p}} + \bar{p}_{\alpha} - i_{\alpha - 1}^{\bar{p}} - \bar{p}_{\alpha - 1}).
\end{equation*}
This quantity can be interpreted by means of Ferrer's diagrams as shown in Example~\ref{ExaFpqGpq}.

\begin{remark}\label{RemFibres}
	If $\Delta_{q}, \Delta_{q'}$ are two $\mathcal{S}$-varieties such that $\Delta_{q^{p}} = \Delta_{q^{\prime p}}$, then $F_{pq} = F_{pq'}$. This fact occurs, for instance, in Example~\ref{ExaAdapt}.
\end{remark}

All spaces and maps defined up to now fit in a cartesian square \cite[Definition 5.1, p. 34]{Ive1986}
\begin{equation}\label{EqCartSquare}
	\begin{tikzcd}
		\Delta_{pq}^{0} \arrow[r, hook, "j_{pq}"] \arrow[d, swap, "\rho_{pq}"] & \tilde{\Delta}_{p} \arrow[d, "\pi_{p}"]\\
		\Delta_{q}^{0} \arrow[r, hook, swap, "i_{pq}^{0}"] & \Delta_{p}
	\end{tikzcd}
\end{equation}
whose horizontal arrows are inclusions and, in particular, $i_{pq}^{0} = i_{pq}|_{\Delta_{q}^{0}}$ is the restriction of the inclusion $i_{pq}: \Delta_{q} \hookrightarrow \Delta_{p}$.

\subsection{Application of the decomposition theorem}\label{SubsecDecTh}

Let $\Delta_{p}$ be an $\mathcal{S}$-variety and let $\Delta_{q}$ be a $\Delta_{p}$-variety. Put
\begin{alignat*}{5}
	m_{p} &:= \dim \Delta_{p}, \qquad && k_{pq} &&:= \dim F_{pq}, \qquad && d_{pq} &&:= m_{p} - m_{q}  - k_{pq},\\
	\delta_{pq} &:= k_{pq} - d_{pq}, \qquad && A_{pq}^{\alpha} &&:= H^{\alpha}(F_{pq}), \qquad && a_{pq}^{\alpha} &&:= \dim_{\mathbb{Q}} A_{pq}^{\alpha}.
\end{alignat*}
From the square \eqref{EqCartSquare} we infer \cite[Formula (15) and Remark 3.1]{Fra2020}
\begin{equation}\label{EqFra01}
	R \pi_{p*} \mathbb{Q}_{\tilde{\Delta}_{p}} \left[ m_{p} \right]|_{\Delta_{q}^{0}} \cong R \rho_{pq*} \mathbb{Q}_{\Delta_{pq}^{0}} \left[ m_{p} \right] \cong \bigoplus_{\alpha = 0}^{2 k_{pq}} A_{pq}^{\alpha} \otimes \mathbb{Q}_{\Delta_{q}^{0}} \left[ m_{p} - \alpha \right]
\end{equation}
and, for any $\alpha \in \mathbb{Z}$,
\begin{equation}\label{EqFra02}
	\prescript{\mathfrak{p}}{}{\mathcal{H}}^{\alpha}(R \pi_{p*} \mathbb{Q}_{\tilde{\Delta}_{p}}|_{\Delta_{q}^{0}}) \cong A_{pq}^{\alpha - m_{q}} \otimes \mathbb{Q}_{\Delta_{q}^{0}} [m_{q}].
\end{equation}

\begin{theorem}(Decomposition theorem \cite[1.6.1]{dCaMi2009})\label{ThDec}
	Let $f: X \rightarrow Y$ be a proper map of complex algebraic varieties. There is an isomorphism in the constructible bounded derived category $D_{c}^{b}(Y)$
	\begin{equation*}
		Rf_{*} IC_{X} \cong \bigoplus_{\alpha \in \mathbb{Z}} \prescript{\mathfrak{p}}{}{\mathcal{H}}^{\alpha}(Rf_{*} IC_{X}) \left[ - \alpha \right].
	\end{equation*}
	Furthermore, the perverse sheaves $\prescript{\mathfrak{p}}{}{\mathcal{H}}^{\alpha}(Rf_{*} IC_{X})$ are semisimple; i.e.~there is a decomposition into finitely many disjoint locally closed and nonsingular subvarieties $Y = \coprod S_{\beta}$ and a canonical decomposition into a direct sum of intersection complexes of semisimple local systems
	\begin{equation*}
		\prescript{\mathfrak{p}}{}{\mathcal{H}}^{\alpha}(Rf_{*} IC_{X}) \cong \bigoplus_{\beta} IC_{\overline{S_{\beta}}}(L_{\beta}).
	\end{equation*}
\end{theorem}

\begin{theorem}\label{ThMain}
	\
	\begin{enumerate}[label = \roman*)]
		\item	For any $\mathcal{S}$-variety $\Delta_{p}$,
		\begin{equation*}
			\prescript{\mathfrak{p}}{}{\mathcal{H}}^{\alpha}(R \pi_{p*} \mathbb{Q}_{\tilde{\Delta}_{p}} [m_{p}]) \cong \bigoplus_{\substack{q \geq p\\ p-adm.}} D_{pq}^{\delta_{pq} + \alpha} \otimes R i_{pq*} IC_{\Delta_{q}}^{\bullet},
		\end{equation*}
		for suitable vector spaces such that $D_{pq}^{\delta_{pq} - \alpha} \cong D_{pq}^{\delta_{pq} + \alpha} \hspace{0.33em} \forall \alpha \geq 0$.
		\item	Given two $\mathcal{S}$-varieties $\Delta_{q} \subseteq \Delta_{p}$,
		\begin{equation*}
			IC_{\Delta_{p}}^{\bullet} [- m_{p}]|_{\Delta_{q}^{0}} \cong \bigoplus_{\alpha \geq 0} B_{pq}^{\alpha} \otimes \mathbb{Q}_{\Delta_{q}^{0}} [- \alpha]
		\end{equation*}
		for suitable vector spaces $B_{pq}^{\alpha}$.
	\end{enumerate}
\end{theorem}

\begin{proof}
	To start with, let us notice that
	\begin{equation*}
		\prescript{\mathfrak{p}}{}{\mathcal{H}}^{\alpha}(R \pi_{p*} \mathbb{Q}_{\tilde{\Delta}_{p}} [m_{p}]) \cong \prescript{\mathfrak{p}}{}{\mathcal{H}}^{- \alpha}(R \pi_{p*} \mathbb{Q}_{\tilde{\Delta}_{p}} [m_{p}]) \hspace{0.33em} \forall \alpha \geq 0
	\end{equation*}
	by virtue of the relative hard Lefschetz theorem \cite[Theorem 1.6.3]{dCaMi2009}. In particular, this implies that $D_{pq}^{\delta_{pq} - \alpha} \cong D_{pq}^{\delta_{pq} + \alpha}$ for any $\alpha \geq 0$.

	We are going to prove \textit{i)} and \textit{ii)} simultaneously by induction.

	\textit{Base step}.
	
	\begin{description}[leftmargin = 5mm, font=\normalfont\itshape]
		\item[i)]	\textit{Assume that $\Delta_{p}$ is minimal in the family of $\mathcal{S}$-varieties}; that is, if $\Delta_{q}$ is an $\mathcal{S}$-variety such that $\Delta_{q} \subseteq \Delta_{p}$, then $\Delta_{q} = \Delta_{p}$.

		$\Delta_{p} = \Delta_{p}^{0}$ is smooth, otherwise there would be a strictly smaller $\mathcal{S}$-variety contained in it. Consequently, $\pi_{p}$ is an isomorphism and
		\begin{equation*}
			R \pi_{p*} \mathbb{Q}_{\tilde{\Delta}_{p}} [m_{p}] \cong \mathbb{Q}_{\Delta_{p}^{0}} [m_{p}] \cong IC_{\Delta_{p}}^{\bullet}|_{\Delta_{p}^{0}} = IC_{\Delta_{p}}^{\bullet},
		\end{equation*}
		where $\mathbb{Q}_{\Delta_{p}^{0}} [m_{p}] \cong IC_{\Delta_{p}}^{\bullet}|_{\Delta_{p}^{0}}$ is \cite[Theorem p. 78, (a)]{GoMa1983}.
		\item[ii)]	\textit{Assume $q = p$}. We have $IC_{\Delta_{p}}^{\bullet}|_{\Delta_{p}^{0}} [- m_{p}] \cong \mathbb{Q}_{\Delta_{p}^{0}}$ by \cite[loc. cit.]{GoMa1983}.
	\end{description}
	
	Before we proceed to the inductive step, let us remind that the perverse cohomology sheaves $\prescript{\mathfrak{p}}{}{\mathcal{H}}^{\alpha}(R \pi_{p*} \mathbb{Q}_{\tilde{\Delta}_{p}} [m_{p}])$ admit a decomposition
	\begin{equation*}
		\prescript{\mathfrak{p}}{}{\mathcal{H}}^{\alpha}(R \pi_{p*} \mathbb{Q}_{\tilde{\Delta}_{p}} [m_{p}]) \cong \bigoplus_{\substack{q \geq p\\ p-adm.}} \prescript{\mathfrak{p}}{}{\mathcal{H}}^{\alpha}(R \pi_{p*} \mathbb{Q}_{\tilde{\Delta}_{p}} [m_{p}])_{\Delta_{q}},
	\end{equation*}
	where $\prescript{\mathfrak{p}}{}{\mathcal{H}}^{\alpha}(R \pi_{p*} \mathbb{Q}_{\tilde{\Delta}_{p}} [m_{p}])_{\Delta_{q}}$ denotes the $\Delta_{q}$-summand in the decomposition by supports \cite[\textsection 1.1]{dCa2013}. Then, in order to prove \textit{i)}, we have to show that, for any $p$-admissible $q$ and any $\alpha \in \mathbb{Z}$, the component supported on $\Delta_{q}^{0}$ is
	\begin{equation}\label{EqDeltaSummand}
		\prescript{\mathfrak{p}}{}{\mathcal{H}}^{\alpha}(R \pi_{p*} \mathbb{Q}_{\tilde{\Delta}_{p}} [m_{p}])_{\Delta_{q}^{0}} \cong D_{pq}^{\delta_{pq} + \alpha} \otimes R i_{pq*}^{0} \mathbb{Q}_{\Delta_{q}^{0}} [m_{q}].
	\end{equation}
	It suffices to prove isomorphism \eqref{EqDeltaSummand} for $\alpha \geq 0$ because of Hard Lefschetz theorem.
	
	\textit{Inductive step}.
	
	\begin{description}[leftmargin = 5mm, font=\normalfont\itshape]
		\item[i)] \textit{Let $q > p$ be $p$-admissible}. By inductive hypothesis, Formula \eqref{EqDeltaSummand} holds for every $p$-admissible $\tau$ such that $p \leq \tau < q$. Therefore, we have to prove that
		\begin{equation*}
			\prescript{\mathfrak{p}}{}{\mathcal{H}}^{\alpha}(R \pi_{p*} \mathbb{Q}_{\tilde{\Delta}_{p}} [m_{p}])_{\Delta_{q}^{0}} \cong D_{pq}^{\delta_{pq} + \alpha} \otimes R i_{pq*}^{0} \mathbb{Q}_{\Delta_{q}^{0}} [m_{q}].
		\end{equation*}
		The decomposition theorem gives (see \eqref{decThmintr} and \eqref{decimpr})
		\begin{align*}
			R \pi_{p*} \mathbb{Q}_{\tilde{\Delta}_{p}} [m_{p}]|_{\Delta_{q}^{0}} &\cong \bigoplus_{\alpha \in \mathbb{Z}} \prescript{\mathfrak{p}}{}{\mathcal{H}}^{\alpha}(R \pi_{p*} \mathbb{Q}_{\tilde{\Delta}_{p}} [m_{p}])|_{\Delta_{q}^{0}} [- \alpha]\\
			&\cong \bigoplus_{\alpha \in \mathbb{Z}} \prescript{\mathfrak{p}}{}{\mathcal{H}}^{\alpha}(R \pi_{p*} \mathbb{Q}_{\tilde{\Delta}_{p}} [m_{p}])_{\Delta_{q}^{0}} [- \alpha]\\
			&\qquad \bigoplus_{\alpha \in \mathbb{Z}} \prescript{\mathfrak{p}}{}{\mathcal{H}}^{\alpha}(R \pi_{p*} \mathbb{Q}_{\tilde{\Delta}_{p}} [m_{p}])_{\Delta_{p}}|_{\Delta_{q}^{0}} [- \alpha]\\
			&\qquad \bigoplus_{\alpha \in \mathbb{Z}} \bigoplus_{\substack{p < \tau < q\\ p-adm.}} \prescript{\mathfrak{p}}{}{\mathcal{H}}^{\alpha}(R \pi_{p*} \mathbb{Q}_{\tilde{\Delta}_{p}} [m_{p}])_{\Delta_{\tau}}|_{\Delta_{q}^{0}} [- \alpha].
		\end{align*}
		By inductive hypothesis, for any $p$-admissible $\tau$ with $p < \tau < q$,
		\begin{align*}
			\prescript{\mathfrak{p}}{}{\mathcal{H}}^{\alpha}(R \pi_{p*} \mathbb{Q}_{\tilde{\Delta}_{p}} [m_{p}])_{\Delta_{\tau}}|_{\Delta_{q}^{0}} &\cong D_{p \tau}^{\delta_{p \tau} + \alpha} \otimes R i_{p \tau*} IC_{\Delta_{\tau}}^{\bullet}|_{\Delta_{q}^{0}}\\
			&\cong D_{p \tau}^{\delta_{p \tau} + \alpha} \otimes i_{qq}^{0*} \circ i_{pq}^{*} \circ R i_{p \tau*} IC_{\Delta_{\tau}}^{\bullet}\\
			&\cong D_{p \tau}^{\delta_{p \tau} + \alpha} \otimes i_{qq}^{0*} \circ i_{\tau q}^{*} \circ i_{p \tau}^{*} \circ R i_{p \tau*} IC_{\Delta_{\tau}}^{\bullet}\\
			&\cong D_{p \tau}^{\delta_{p \tau} + \alpha} \otimes IC_{\Delta_{\tau}}^{\bullet}|_{\Delta_{q}^{0}}\\
			&\cong \bigoplus_{\beta \geq 0} D_{p \tau}^{\delta_{p \tau} + \alpha} \otimes \left( B_{\tau q}^{\beta} \otimes \mathbb{Q}_{\Delta_{q}^{0}} [m_{\tau} - \beta] \right),
		\end{align*}
		where we used functoriality and exactness of the pullback and the fact that $i_{p \tau}^{*} \circ i_{p \tau *} = 1$ \cite[p. 110]{Ive1986}. If we substitute this in the preceding isomorphism and combine it with Formula \eqref{EqFra01}, we obtain
		\begin{align*}
			\bigoplus_{\alpha = 0}^{2 k_{pq}} A_{pq}^{\alpha + m_{p}} \otimes \mathbb{Q}_{\Delta_{q}^{0}} [- \alpha] &\cong \bigoplus_{\alpha \in \mathbb{Z}} \prescript{\mathfrak{p}}{}{\mathcal{H}}^{\alpha}(R \pi_{p*} \mathbb{Q}_{\tilde{\Delta}_{p}} [m_{p}])_{\Delta_{q}^{0}} [- \alpha]\\
			&\qquad \bigoplus_{\alpha \in \mathbb{Z}} R^{\alpha}(IC_{\Delta_{p}}^{\bullet}|_{\Delta_{q}^{0}}) [- \alpha]\\
			&\qquad \bigoplus_{\substack{p < \tau < q\\ p-adm.}} \bigoplus_{\substack{\alpha \in \mathbb{Z}\\ \beta \geq 0}} D_{p \tau}^{\delta_{p \tau} + \alpha} \otimes B_{\tau q}^{\beta + m_{\tau}} \otimes \mathbb{Q}_{\Delta_{q}^{0}} [- \alpha - \beta].
		\end{align*}
		It follows that, for every fixed $\gamma \in \mathbb{Z}$,
		\begin{equation}\label{EqProofMainTh}
			\begin{split}
			A_{pq}^{\gamma + m_{p}} \otimes \mathbb{Q}_{\Delta_{q}^{0}} \cong &\prescript{\mathfrak{p}}{}{\mathcal{H}}^{\gamma + m_{q}}(R \pi_{p*} \mathbb{Q}_{\tilde{\Delta}_{p}} [m_{p}])_{\Delta_{q}^{0}}\\
			&\qquad \oplus R^{\gamma}(IC_{\Delta_{p}}^{\bullet}|_{\Delta_{q}^{0}})\\
			&\qquad \bigoplus_{\alpha + \beta = \gamma} \bigoplus_{\substack{p < \tau < q\\ p-adm.}} D_{p \tau}^{\delta_{p \tau} + \alpha} \otimes B_{\tau q}^{\beta + m_{\tau}} \otimes \mathbb{Q}_{\Delta_{q}^{0}}.
			\end{split}
		\end{equation}
		Remember that we want to prove Equation~\eqref{EqDeltaSummand} for non-negative exponents; that is, for any $\gamma \geq - m_{q}$. For such integers, $R^{\gamma}(IC_{\Delta_{p}}^{\bullet}|_{\Delta_{q}^{0}}) = 0$ because the intersection cohomology complexes satisfy support conditions \cite[\textsection 2.1]{dCaMi2009} and, as a consequence, isomorphism~\eqref{EqProofMainTh} becomes
		\begin{align*}
			A_{pq}^{\gamma + m_{p}} \otimes \mathbb{Q}_{\Delta_{q}^{0}} &\cong \prescript{\mathfrak{p}}{}{\mathcal{H}}^{\gamma + m_{q}}(R \pi_{p*} \mathbb{Q}_{\tilde{\Delta}_{p}} [m_{p}])_{\Delta_{q}^{0}}\\
			&\qquad \bigoplus_{\alpha + \beta = \gamma} \bigoplus_{\substack{p < \tau < q\\ p-adm.}} D_{p \tau}^{\delta_{p \tau} + \alpha} \otimes B_{\tau q}^{\beta + m_{\tau}} \otimes \mathbb{Q}_{\Delta_{q}^{0}}.
		\end{align*}
		In particular, it follows that $\prescript{\mathfrak{p}}{}{\mathcal{H}}^{\gamma + m_{q}}(R \pi_{p*} \mathbb{Q}_{\tilde{\Delta}_{p}} [m_{p}])_{\Delta_{q}^{0}}$ is a $\Delta_{q}^{0}$-trivial local system and, therefore, there are suitable vector spaces for which \eqref{EqDeltaSummand} holds.
		\item[ii)]	\textit{Let $\Delta_{q} \subset \Delta_{p}$ be an $\mathcal{S}$-variety}.

		If $\Delta_{q} = \Delta_{q^{p}}$, $\Delta_{q}$ is a $\Delta_{p}$-variety and Formula \eqref{EqProofMainTh} holds. Since \textit{i)} has been proved, it can be written as follows (notice that $\Delta_{q}^{0} = \Delta_{q^{p}}^{0}$ and $m_{q} = m_{\bar{q}}$):
		\begin{align*}
			A_{pq}^{\gamma + m_{p}} \otimes \mathbb{Q}_{\Delta_{q}^{0}} &\cong D_{pq}^{\delta_{pq} + \gamma + m_{q}} \otimes \mathbb{Q}_{\Delta_{q}^{0}} \oplus R^{\gamma}(IC_{\Delta_{p}}^{\bullet}|_{\Delta_{q}^{0}})\\
			&\qquad \bigoplus_{\alpha + \beta = \gamma} \bigoplus_{\substack{p < \tau < q\\ p-adm.}} D_{p \tau}^{\delta_{p \tau} + \alpha} \otimes B_{\tau q}^{\beta + m_{\tau}} \otimes \mathbb{Q}_{\Delta_{q}^{0}}.
		\end{align*}
		In particular, $R^{\gamma}(IC_{\Delta_{p}}^{\bullet}|_{\Delta_{q}^{0}})$ is a trivial local system on $\Delta_{q}^{0}$ and the assertion follows from \cite[Remark 1.5.1]{dCaMi2009}.

		If $\Delta_{q} \neq \Delta_{q^{p}}$, take $\Delta_{q^{p}}$, which is a $\Delta_{p}$-variety. $\Delta_{q}^{0}$ is strictly contained in $\Delta_{q^{p}}^{0}$, thus, from what has just been proved, it follows that
		\begin{equation*}
			IC_{\Delta_{p}}^{\bullet} [- m_{p}]|_{\Delta_{q}^{0}} \cong IC_{\Delta_{p}}^{\bullet} [- m_{p}]|_{\Delta_{q^{p}}^{0}} |_{\Delta_{q}^{0}} \cong \bigoplus_{\alpha \geq 0} B_{pq^{p}}^{\alpha} \otimes \mathbb{Q}_{\Delta_{q}^{0}} [- \alpha]. \qedhere
		\end{equation*}
	\end{description}
\end{proof}

It is worth, and also useful, pointing out that in the proof of Theorem~\ref{ThMain}~(ii) we found out the following result.

\begin{corollary}\label{CorBpq}
	Given two $\mathcal{S}$-varieties $\Delta_{q} \subseteq \Delta_{p}$, the vector spaces $B_{pq}^{\alpha}$ and $B_{pq^p}^{\alpha}$ appearing in the decompositions of $IC_{\Delta_{p}}^{\bullet} [- m_{p}]|_{\Delta_{q}^{0}}$ and $IC_{\Delta_{p}}^{\bullet}[- m_{p}]|_{\Delta_{q^{p}}^{0}}$, respectively, are the same.
\end{corollary}

In Theorem~\ref{ThMain}~(i), we took $\Delta_{q}$ to be a $\Delta_{p}$-variety because the resolution $\pi_{p}$ takes into account only the necessary conditions of $\Delta_{p}$; in other words, we restricted to the family of $\Delta_{p}$-varieties. Nevertheless, there is something we can say even when $\Delta_{q}$ is not a $\Delta_{p}$-variety.

\begin{proposition}[Enhancement of Formula \eqref{EqDeltaSummand}]\label{PropDpq}
	If $\Delta_{q} \subseteq \Delta_{p}$ are two $\mathcal{S}$-varieties,
	\begin{equation*}
		\prescript{\mathfrak{p}}{}{\mathcal{H}}^{\alpha}(R \pi_{p*} \mathbb{Q}_{\tilde{\Delta}_{p}} [m_{p}])_{\Delta_{q^{p}}}|_{\Delta_{q}^{0}} \cong D_{pq^{p}}^{\delta_{pq^{p}} + \alpha} \otimes \mathbb{Q}_{\Delta_{q}^{0}} [m_{q^{p}}].
	\end{equation*}
\end{proposition}
\begin{proof}
	$\Delta_{q}^{0} \subseteq \Delta_{q^{p}}^{0}$, thus Formula \eqref{EqDeltaSummand} applied to $q^{p}$ provides
	\vskip1mm
	\begin{minipage}{0.5\linewidth}
		\begin{align*}
			&\prescript{\mathfrak{p}}{}{\mathcal{H}}^{\alpha}(R \pi_{p*} \mathbb{Q}_{\tilde{\Delta}_{p}} [m_{p}])_{\Delta_{q^{p}}}|_{\Delta_{q}^{0}}\\
			&\qquad= \prescript{\mathfrak{p}}{}{\mathcal{H}}^{\alpha}(R \pi_{p*} \mathbb{Q}_{\tilde{\Delta}_{p}} [m_{p}])_{\Delta_{q^{p}}}|_{\Delta_{q^{p}}^{0}} |_{\Delta_{q}^{0}}\\
			&\qquad \cong D_{pq^{p}}^{\delta_{pq^{p}} + \alpha} \otimes \mathbb{Q}_{\Delta_{q^{p}}^{0}} [m_{q^{p}}]|_{\Delta_{q}^{0}}\\
			&\qquad \cong D_{pq^{p}}^{\delta_{pq^{p}} + \alpha} \otimes \mathbb{Q}_{\Delta_{q}^{0}} [m_{q^{p}}].
		\end{align*}
	\end{minipage}
	\begin{minipage}{0.45\linewidth}
		\begin{equation*}
			\begin{tikzcd}
				& \Delta_{p} & \\
				\Delta_{q}^{0} \arrow[ur, hook] \arrow[r, hook] \arrow[dr, hook] & \Delta_{q} \arrow[u, hook, dashed] \arrow[r, hook] & \Delta_{q^{p}} \arrow[ul, hook] \\
				& \Delta_{q^{p}}^{0} \arrow[ur, hook] &
			\end{tikzcd}
		\end{equation*}
	\end{minipage}
	\vskip1mm
\end{proof}

\begin{notation}
	When $\Delta_{q}$ is not a $\Delta_{p}$-variety, the Proposition~\ref{PropDpq} allows us to put $D_{pq}^{\alpha} := D_{pq^{p}}^{\alpha}$.
\end{notation}

\subsection{Polynomial expressions}\label{SubsecPolExpr}

We are going to deal with a peculiar family of polynomials, which require a suitable notation.

Given a topological space $X$, the \textit{Poincaré polynomial of $X$} shall be denoted by
\begin{equation*}
	H_{X} := \sum_{\alpha \geq 0} \dim H^{\alpha}(X).
\end{equation*}
In particular, when $X = \mathbb{G}_{k}(\mathbb{C}^{l})$, it is known that \cite[\textsection 5.2]{ChGoMa1982}
\begin{equation*}
	H_{\mathbb{G}_{k}(\mathbb{C}^{l})} = \frac{P_{l}}{P_{k} P_{l - k}},
\end{equation*}
where
\begin{equation*}
	P_{\alpha} :=
	\begin{cases*}
		0 &\mbox{if $\alpha < 0$}\\
		1 &\mbox{if $\alpha = 0$}\\
		h_{0} \ldots h_{\alpha - 1} &\mbox{if $\alpha > 0$}
	\end{cases*}
	\quad \mbox{and} \quad h_{\beta} := \sum_{\alpha = 0}^{\beta} t^{2 \alpha} \hspace{0.33em} \forall \beta \in \mathbb{Z}.
\end{equation*}

Let $\Delta_{p}$ and $\Delta_{q}$ be two $\mathcal{S}$-varieties. Set
\begin{alignat*}{3}
	a_{pq} &:= \sum_{\alpha \in \mathbb{Z}} a_{pq}^{\alpha} t^{\alpha} = \sum_{\alpha \in \mathbb{Z}} \dim_{\mathbb{Q}} A_{pq}^{\alpha} t^{\alpha}, \qquad && f_{pq} &&:= \sum_{\alpha \in \mathbb{Z}} d_{pq}^{\alpha} t^{\alpha} = \sum_{\alpha \in \mathbb{Z}} \dim_{\mathbb{Q}} D_{pq}^{\alpha} t^{\alpha}\\
	g_{pq} &:= f_{pq} t^{2d_{pq}} = \sum_{\alpha \in \mathbb{Z}} \dim_{\mathbb{Q}} D_{pq}^{\alpha} t^{\alpha + 2d_{pq}}, \qquad && b_{pq} &&:= \sum_{\alpha \in \mathbb{Z}} b_{pq}^{\alpha} t^{\alpha} = \sum_{\alpha \in \mathbb{Z}} \dim_{\mathbb{Q}} B_{pq}^{\alpha} t^{\alpha}.
\end{alignat*}
The polynomials $f_{pq}$ (and, consequently, $g_{pq}$) are well defined because of Proposition~\ref{PropDpq}. On the other hand, Remark~\ref{RemFibres} and Corollary~\ref{CorBpq} guarantee that, if $q$ is not $p$-admissible, $a_{pq} = a_{p q^{p}}$ and $b_{pq} = b_{pq^{p}}$, respectively. Furthermore, the polynomials $b_{pq}$ are (some) Kazhdan-Lusztig polynomials (see \cite[Ch. 6]{BiLa2000}, \cite[\textsection 4.4]{dCaMi2009}).

\begin{remark}[Multiplicities $n_{hk}$]\label{RemMultiplicities}
	In the introduction we stated in Formula \eqref{decimpr} that $R \pi_{*} \mathbb{Q}_{\tilde{\mathcal{S}}} \cong \oplus_{h, k} IC_{\Delta_{h}}^{\bullet} [-k]^{\oplus n_{hk}}$. We are now able to be more accurate: if we combine decomposition theorem~\ref{ThDec} with Theorem~\ref{ThMain} we obtain
	\begin{align*}
		R \pi_{0*} \mathbb{Q}_{\tilde{\mathcal{S}}} &\cong \bigoplus_{\substack{\alpha \in \mathbb{Z}\\ adm. \hspace{0.165em} q}} D_{0q}^{\delta_{0q} + \alpha} \otimes R i_{0q*} IC_{\Delta_{q}}^{\bullet} [- \dim \mathcal{S} - \alpha]\\
		&\cong \bigoplus_{\substack{\alpha \in \mathbb{Z}\\ adm. \hspace{0.165em} q}} R i_{0q*} IC_{\Delta_{q}}^{\bullet} [- \dim \mathcal{S} - \alpha]^{\oplus d_{0q}^{\alpha}}.
	\end{align*}
	As a consequence, in Formula \eqref{decimpr} we have $h \in \{ q: \Delta_{q} \mbox{ is an } \mathcal{S}\mbox{-variety} \}$, $k = \dim \mathcal{S} + \alpha$ and $n_{hk} = d_{0h}^{\alpha}$, up to identifying $IC_{\Delta_{q}}^{\bullet}$ with its (derived) direct image under the inclusion $i_{0h}$.
\end{remark}

Among the above polynomials, the ones which are \textit{always explicit} are the $a_{pq}$; namely,

\begin{proposition}\label{Prop_apq}
	\begin{equation*}
		a_{pq} = \frac{P_{i_{1}^{\bar{p}} + q_{1}^{p}}}{P_{i_{1}^{\bar{p}} + \bar{p}_{1}} P_{q_{1}^{p} - \bar{p}_{1}}} \cdot \prod_{\alpha = 2}^{\omega_{\bar{p}}} \frac{P_{i_{\alpha}^{\bar{p}} + q_{\alpha}^{p} - i_{\alpha - 1}^{\bar{p}} - \bar{p}_{\alpha - 1}}}{P_{i_{\alpha}^{\bar{p}} + \bar{p}_{\alpha} - i_{\alpha - 1}^{\bar{p}} - \bar{p}_{\alpha - 1}} P_{q_{\alpha}^{p} - \bar{p}_{\alpha}}}.
	\end{equation*}
\end{proposition}
\begin{proof}
	Recall that $a_{pq} = \sum_{\alpha} \dim H^{\alpha}(F_{pq}) = H_{F_{pq}}$. In order to make notations simpler, put, for any $\alpha \in \{ 1, \ldots, \omega_{\bar{p}} \}$,
	\begin{equation*}
		\mathbb{F}^{\alpha} := \mathbb{F}(i_{1}^{\bar{p}} + \bar{p}_{1}, \ldots, i_{\alpha}^{\bar{p}} + \bar{p}_{\alpha}; \mathbb{C}^{i_{1}^{\bar{p}} + q_{1}^{p}}, \ldots, \mathbb{C}^{i_{\alpha}^{\bar{p}} + q_{\alpha}^{p}}).
	\end{equation*}
	For any $\alpha > 1$, the projection $\mathbb{F}^{\alpha} \rightarrow \mathbb{F}^{\alpha - 1}$ is a fibration with fibres
	\begin{equation*}
	\mathbb{G}_{i_{\alpha}^{\bar{p}} + \bar{p}_{\alpha} - i_{\alpha - 1}^{\bar{p}} - \bar{p}_{\alpha - 1}}(\mathbb{C}^{i_{\alpha}^{\bar{p}} + q_{\alpha}^{p} - i_{\alpha - 1}^{\bar{p}} - \bar{p}_{\alpha - 1}}).
	\end{equation*}
	By Leray-Hirsch theorem,
	\begin{align*}
		H^{\bullet}(F_{pq}) &\cong H^{\bullet}(\mathbb{G}_{i_{1}^{\bar{p}} + \bar{p}_{1}}(\mathbb{C}^{i_{1}^{\bar{p}} + q_{1}^{p}})) \otimes H^{\bullet}(\mathbb{G}_{i_{2}^{\bar{p}} + \bar{p}_{2} - i_{1}^{\bar{p}} - \bar{p}_{1}}(\mathbb{C}^{i_{2}^{\bar{p}} + q_{2}^{\bar{p}} - i_{1}^{\bar{p}} - \bar{p}_{1}}))\\
		&\qquad \otimes \ldots \otimes H^{\bullet}(\mathbb{G}_{i_{\omega_{\bar{p}}}^{\bar{p}} + \bar{p}_{\omega_{\bar{p}}} - i_{\omega_{\bar{p}} - 1}^{\bar{p}} - \bar{p}_{\omega_{\bar{p}} - 1}}(\mathbb{C}^{i_{\omega_{\bar{p}}}^{\bar{p}} + q_{\omega_{\bar{p}}}^{\bar{p}} - i_{\omega_{\bar{p}} - 1}^{\bar{p}} - \bar{p}_{\omega_{\bar{p}} - 1}})).
	\end{align*}
	The assertion follows by taking the Poincaré polynomials.
\end{proof}

We are now going to exhibit the existence of a family of (Poincaré) polynomial expressions by means of Theorem~\ref{ThMain}.

\begin{corollary}\label{CorPolExp}
	Suppose that $\Delta_{q}$ is a $\Delta_{p}$-variety. If $\Delta_{p} = \Delta_{q}$,
	\begin{equation*}
		a_{pp} = g_{pp} = b_{pp} = 1,
	\end{equation*}
	otherwise
	\begin{align*}
		a_{pq} = b_{pq} + g_{pq} + \sum_{\substack{p < \tau < q\\ p-adm.}} g_{p \tau} b_{\tau q}.
	\end{align*}
\end{corollary}

\begin{proof}
	Assume $p \neq q$. In Theorem~\ref{ThMain} we proved that
	\begin{equation*}
		\prescript{\mathfrak{p}}{}{\mathcal{H}}^{\alpha} (R \pi_{p*} \mathbb{Q}_{\tilde{\Delta}_{p}} [m_{p}] \mid_{\Delta_{q}^{0}}) \cong \bigoplus_{\substack{p \leq \tau \leq q\\ p-adm.}} D_{p \tau}^{\delta_{p \tau} + \alpha} \otimes R i_{p \tau*} IC_{\Delta_{\tau}}^{\bullet}\mid_{\Delta_{q}^{0}}
	\end{equation*}
and
	\begin{equation*}
		IC_{\Delta_{\tau}}^{\bullet}|_{\Delta_{q}^{0}} \cong \bigoplus_{\beta \in \mathbb{Z}} B_{\tau q}^{\beta} \otimes \mathbb{Q}_{\Delta_{q}^{0}} [m_{\tau} - \beta] \cong \bigoplus_{\beta \in \mathbb{Z}} B_{\tau q}^{\beta + m_{\tau}} \otimes \mathbb{Q}_{\Delta_{q}^{0}} [- \beta].
	\end{equation*}
	Combining these results with Formula \eqref{EqFra02}, we obtain
	\begin{align*}
		&\bigoplus_{\alpha \in \mathbb{Z}} A_{pq}^{\alpha + m_{p} - m_{q}} \otimes \mathbb{Q}_{\Delta_{q}^{0}} [m_{q} - \alpha]\\
		&\qquad \cong \bigoplus_{\alpha \in \mathbb{Z}} \left( \bigoplus_{\substack{p \leq \tau \leq q\\ p-adm.}} D_{p \tau}^{\delta_{p \tau} + \alpha} \otimes \bigoplus_{\beta \in \mathbb{Z}} B_{\tau q}^{\beta + m_{\tau}} \otimes \mathbb{Q}_{\Delta_{q}^{0}} [- \beta] \right) [- \alpha],
	\end{align*}
	which can be written as
	\begin{equation*}
		\bigoplus_{\alpha \geq - m_{p}} A_{pq}^{\alpha + m_{p}} \otimes \mathbb{Q}_{\Delta_{q}^{0}} [- \alpha] \cong \bigoplus_{\alpha, \beta \geq - m_{p}} \bigoplus_{\substack{p \leq \tau \leq q\\ p-adm.}} D_{p \tau}^{\delta_{p \tau} + \alpha} \otimes B_{\tau q}^{\beta + m_{\tau}} \otimes \mathbb{Q}_{\Delta_{q}^{0}} [- \alpha - \beta].
	\end{equation*}
	For any $\gamma \geq - m_{p}$,
	\begin{equation}\label{EqCorMain}
		A_{pq}^{\gamma + m_{p}} \otimes \mathbb{Q}_{\Delta_{q}^{0}} \cong \bigoplus_{\alpha + \beta = \gamma} \bigoplus_{\substack{p \leq \tau \leq q\\ p-adm.}} D_{p \tau}^{\delta_{p \tau} + \alpha} \otimes B_{\tau q}^{\beta + m_{\tau}} \otimes \mathbb{Q}_{\Delta_{q}^{0}}.
	\end{equation}
	Let us see what happens for $\tau = p$ and $\tau = q$. When $\tau = p$, $\delta_{pp} = 0$ and
	\begin{equation*}
		D_{pp}^{\alpha} = \begin{cases*} \mathbb{Q} &\mbox{if $\alpha = 0$,}\\ 0 &\mbox{otherwise.} \end{cases*}
	\end{equation*}
	Therefore
	\begin{equation*}
		\bigoplus_{\alpha + \beta = \gamma} D_{pp}^{\delta_{pp} + \alpha} \otimes B_{pq}^{\beta + m_{p}} \otimes \mathbb{Q}_{\Delta_{q}^{0}} \cong B_{pq}^{\gamma + m_{p}} \otimes \mathbb{Q}_{\Delta_{q}^{0}}.
	\end{equation*}
	On the other hand, when $\tau = q$,
	\begin{equation*}
		B_{qq}^{\beta + m_{q}} = \begin{cases*} \mathbb{Q} &\mbox{if $\beta = - m_{q}$,}\\ 0 &\mbox{otherwise.} \end{cases*}
	\end{equation*}
	Therefore
	\begin{equation*}
		\bigoplus_{\alpha + \beta = \gamma} D_{pq}^{\delta_{pq} + \alpha} \otimes B_{qq}^{\beta + m_{q}} \otimes \mathbb{Q}_{\Delta_{q}^{0}} \cong D_{pq}^{\delta_{pq} + \gamma + m_{q}} \otimes \mathbb{Q}_{\Delta_{q}^{0}}.
	\end{equation*}
	Taking into account these facts, isomorphism \eqref{EqCorMain} becomes
	\begin{align*}
		A_{pq}^{\gamma + m_{p}} \otimes \mathbb{Q}_{\Delta_{q}^{0}} &\cong B_{pq}^{\gamma + m_{p}} \otimes \mathbb{Q}_{\Delta_{q}^{0}} \oplus D_{pq}^{\delta_{pq} + \gamma + m_{q}} \otimes \mathbb{Q}_{\Delta_{q}^{0}}\\
		&\qquad \bigoplus_{\alpha + \beta = \gamma} \bigoplus_{\substack{p < \tau < q\\ p-adm.}} D_{p \tau}^{\delta_{p \tau} + \alpha} \otimes B_{\tau q}^{\beta + m_{\tau}} \otimes \mathbb{Q}_{\Delta_{q}^{0}}.
	\end{align*}
	Let us change notations: put $s = \gamma + m_{p} \hspace{1mm} (\geq 0)$ and use the equality $m_{p} - m_{\tau} - \delta_{p \tau} = 2d_{\bar{p} \tau}$ so as to have
	\begin{align*}
		A_{pq}^{s} \otimes \mathbb{Q}_{\Delta_{q}^{0}} &\cong B_{pq}^{s} \otimes \mathbb{Q}_{\Delta_{q}^{0}} \oplus D_{pq}^{s - 2d_{pq}} \otimes \mathbb{Q}_{\Delta_{q}^{0}}\\
		&\qquad \bigoplus_{\alpha + \beta = s} \bigoplus_{\substack{p < \tau < q\\ p-adm.}} D_{p \tau}^{\alpha - 2d_{\bar{p} \tau}} \otimes B_{\tau q}^{\beta} \otimes \mathbb{Q}_{\Delta_{q}^{0}}.
	\end{align*}
	From this formula we infer, for any $s \geq 0$,
	\begin{equation*}
		a_{pq}^{s} = b_{pq}^{s} + d_{pq}^{s - 2d_{pq}} + \sum_{\alpha + \beta = s} \sum_{\substack{p < \tau < q\\ p-adm.}} d_{\bar{p} \tau}^{\alpha - 2d_{\bar{p} \tau}} b_{\tau q}^{\beta},
	\end{equation*}
	where
	\begin{equation*}
		a_{pq}^{s} = \dim_{\mathbb{Q}} A_{pq}^{s}, \qquad b_{pq}^{s} = \dim_{\mathbb{Q}} B_{pq}^{s}, \qquad d_{pq}^{s} = \dim_{\mathbb{Q}} D_{pq}^{s}.
	\end{equation*}
	If we formally multiply both sides by $t^{s}$ and take the sum over $s$, we obtain the polynomial expression
	\begin{align*}
		a_{pq} &= \sum_{s \geq 0} a_{pq}^{s} t^{s} = \sum_{s \geq 0} b_{pq}^{s} t^{s} + \sum_{s \geq 0} (d_{pq}^{s - 2d_{pq}} t^{s - 2d_{pq}}) t^{2d_{pq}}\\
		&\qquad + \sum_{\substack{p < \tau < q\\ p-adm.}} \left( \sum_{\alpha \geq 0} d_{\bar{p} \tau}^{\alpha - 2d_{\bar{p} \tau}} t^{\alpha - 2d_{\bar{p} \tau}} \right) \left( \sum_{\beta \geq 0} b_{\tau q}^{\beta} t^{\beta} \right) t^{2d_{\bar{p} \tau}}\\
		&= b_{pq} + f_{pq} t^{2d_{pq}} + \sum_{\substack{p < \tau < q\\ p-adm.}} f_{p \tau} b_{\tau q} t^{2d_{\bar{p} \tau}} = b_{pq} + g_{pq} + \sum_{\substack{p < \tau < q\\ p-adm.}} g_{p \tau} b_{\tau q},
	\end{align*}
	where
	\begin{equation*}
		a_{pq} = \sum_{s \geq 0} a_{pq}^{s} t^{s}, \qquad b_{pq} = \sum_{s \geq 0} b_{pq}^{s} t^{s} \qquad f_{pq} = \sum_{s \geq 0} d_{pq}^{s} t^{s}, \qquad g_{pq} = f_{pq} t^{2d_{pq}}.
	\end{equation*}
	Now, consider the case $p = q$. We have $a_{pp} = 1$ either by Proposition~\ref{Prop_apq} or by the fact that $\pi_{p}: \pi_{p}^{-1}(\Delta_{p}^{0}) \rightarrow \Delta_{p}^{0}$ is an isomorphism. Moreover, $g_{pp} = 1$ being $D_{pp}^{\alpha} = \mathbb{Q}$ for $\alpha = 0$ and 0 otherwise (as we saw in the preceding case). Lastly, $b_{pp} = 1$, as well, because all but the first coefficients of the Kazhdan-Lusztig polynomial are 0, being $IC_{\Delta_{p}}^{\bullet}|_{\Delta_{p}^{0}} [-m_{p}] \cong \mathbb{Q}_{\Delta_{p}^{0}}$.
\end{proof}

\begin{remark}
	In Corollary~\ref{CorPolExp} it is not possible to get rid of the hypothesis of $p$-admissibility. Having said that, the above formula works for all $\mathcal{S}$-varieties in the sense that it is legitimate to replace $\Delta_{q}$ with $\Delta_{q^{p}}$.
\end{remark}

\section{Computation of certain Poincaré polynomials}\label{SecKaLu}

In this final section, we are going to use the theoretic results seen up to now so as to obtain an iterative algorithm, which we shall often refer to as \textit{KaLu}, for the computation of the polynomials $g_{pq}$ and $b_{pq}$ (see Section~\ref{SubsecAlgo}). Classes of polynomial identities, which can be used to test \textit{KaLu}, are exhibited in Section~\ref{SubsecTests}. In the end, in Section~\ref{SubsecRelevant}, we see that not all $\mathcal{S}$-varieties contribute to decomposition theorem.

\subsection{KaLu, the iterative algorithm}\label{SubsecAlgo}

\textit{From now on, all $\mathcal{S}$-varieties are supposed to be described with respect to the flag $\mathcal{F}$ of $\mathcal{S}$, unless otherwise stated.}

In Corollary~\ref{CorPolExp} we proved the existence of some classes of polynomial expressions
\begin{equation}\label{EqGpqBpq}
	g_{pq} + b_{pq} = R_{pq},
\end{equation}
where
\begin{equation}\label{EqRpq}
	R_{pq} = a_{pq} - \sum_{\substack{p < \tau < q\\ p-adm.}} g_{p \tau} b_{\tau q}.
\end{equation}
The fact that the Poincaré polynomials $a_{pq}$ are explicit (see Proposition~\ref{Prop_apq}) is fundamental in the achievement of \textit{KaLu}, shown below, for the computation of the polynomials $g_{pq}$ and $b_{pq}$. By the way, let us introduce the following functions:
\begin{align*}
	U_{\beta}&: \sum_{\alpha \geq 0} c_{\alpha} t^{\alpha} \in \mathbb{Z} \left[ t \right] \mapsto \sum_{\alpha \geq \beta} c_{\alpha} t^{\alpha} \in \mathbb{Z} \left[ t \right] \mbox{0.33em} \forall \beta \geq 0,\\
	S&: \sum_{\alpha \geq 0} c_{\alpha} t^{\alpha} \in \mathbb{Z} \left[ t \right] \mapsto c_{0} + \sum_{\alpha \geq 1} c_{\alpha} (t^{\alpha} + t^{- \alpha}) \in \mathbb{Z} \left[ t, t^{-1} \right]\\
	\tilde{t}^{\beta}&: \sum_{\alpha \geq 0} c_{\alpha} t^{\alpha} \in \mathbb{Z} \left[ t \right] \mapsto \sum_{\alpha \geq 0} c_{\alpha} t^{\alpha + \beta} \in \mathbb{Z} \left[ t \right] \hspace{0.33em} \forall \beta \geq 0.
\end{align*}

\begin{corollary}\label{cor:Utilde}
	If $\sum_{\alpha = 1}^{\omega} (q_{\alpha} - p_{\alpha}) \geq 1$,
	\begin{equation*}
		\begin{cases*}
			g_{pq} = \tilde{U}_{pq}(R_{pq})\\
			b_{pq} = R_{pq} - g_{pq}
		\end{cases*}
	\end{equation*}
	where $\tilde{U}_{pq} := \tilde{t}^{m_{pq}} \circ S \circ \tilde{t}^{- m_{pq}} \circ U_{m_{pq}}$ and $m_{pq} := m_{p} - m_{q}$.
\end{corollary}
\begin{proof}
	Theorem~\ref{ThMain} states that the vector spaces $D_{pq}^{\alpha}$ are symmetric with respect to the exponent $\delta_{pq}$; that is, the polynomials $f_{pq}$ are symmetric with respect to the degree $\delta_{pq}$. Since $g_{pq} = f_{pq} t^{2d_{pq}}$, these polynomials are symmetric, as well, but with respect to the degree $m_{pq} := 2d_{pq} + \delta_{pq} = m_{p} - m_{q}$. On the other hand, the vector spaces $B_{pq}^{\alpha}$ were obtained by studying the intersection cohomology complexes $IC_{\Delta_{p}}^{\bullet}$ locally; that is, their restrictions to the locally closed subsets $\Delta_{q}^{0}$. Being $IC_{\Delta_{p}}^{\bullet}$ a perverse sheaf, it satisfies, in particular, the support conditions (see \cite[\textsection 2.3]{dCaMi2009}, \cite[Definition 5.1.11, Proposition 5.1.16]{Dim2004}), thus $B_{pq}^{\alpha} = 0$ for any $\alpha \geq m_{pq}$.

	Assume that $g_{p \tau}$ and $b_{\tau q}$ are known for any $p < \tau < q$; in other words, suppose that $R_{pq}$ is known (when $\sum_{\alpha = 1}^{\omega} (q_{\alpha} - p_{\alpha}) = 1$, $R_{pq} = a_{pq}$ is given by Proposition~\ref{Prop_apq}). $g_{pq}$ can be obtained by $R_{pq}$ by deleting all terms of degree $< m_{pq}$ and by making the new polynomial symmetric with respect to the term of degree $m_{pq}$. Formally, $g_{pq} = \tilde{U}_{pq}(R_{pq})$: indeed
	\begin{enumerate}
		\item	the function $U_{m_{pq}}$ deletes the terms of degree $< m_{pq}$ of $R_{pq}$;
		\item	the function $\tilde{t}^{- m_{pq}}$ is just multiplication by $t^{- m_{pq}}$;
		\item	the function $S$ makes the obtained polynomial symmetric with respect to the term of degree 0;
		\item	the function $\tilde{t}^{m_{pq}}$ shifts the polynomial so as to make it symmetric with respect to the term of degree $m_{pq}$.
	\end{enumerate}
	Finally, $b_{pq}$ is obtained by Equation \eqref{EqGpqBpq}; namely, $b_{pq} = R_{pq} - g_{pq}$.
\end{proof}

The number $m_{pq}$, which is nothing but the codimension of $\Delta_{q}$ in $\Delta_{p}$, plays an important role in \textit{KaLu} because of Corollary~\ref{cor:Utilde}, and is easily calculated if we interpret it by means of Ferrer's diagrams. In fact, from what we observed at the end of Section~\ref{SubsecFerrersdiagrams}, it follows that \textit{$m_{pq}$ is exactly the area of the region between the Ferrer's diagrams of $\Delta_{p}$ and $\Delta_{q}$}. In formulas,
\begin{equation*}
	m_{pq} = m_{p} - m_{q} = \sum_{\alpha = 1}^{k} (\lambda_{\alpha}^{q} - \lambda_{\alpha}^{p}).
\end{equation*}

Thanks to Proposition~\ref{Prop_apq} and Corollaries~\ref{CorPolExp} and \ref{cor:Utilde}, we obtain the algorithm \textit{KaLu}$\big(p,\mathcal I,\mathcal J,k,l,q\big)$ (see Algorithm~\ref{algorithm}), which will be referred to as \textit{KaLu}, for computing the polynomials $b_{pq}$. An implementation of \textit{KaLu} in CoCoA5 is available \url{http://wpage.unina.it/carmine.sessa2/KaLu}.

\begin{algorithm}[!ht]
	\caption{\label{algorithm} Algorithm for computing the polynomial  $b_{pq}$, given $p \leq q$. It also computes all the polynomials $b_{\tau \eta^{\tau}}$, with both $\tau$ and $\eta$ $p$-admissible and $\tau < \eta$, and $g_{\tau \eta}$ if $\eta$ is $\tau$-admissible, as well.}
	\begin{algorithmic}[1]
		\STATE KaLu$\big(p,\mathcal I,\mathcal J,k,l,q\big)$
		\REQUIRE $p,\mathcal I,\mathcal J,q$ vectors of integers of the same length and $k,l$ integers such that $\mathcal I,\mathcal J,k,l$  satisfy conditions \eqref{eq:weak conditions} (hence, determine a Schubert variety $\mathcal S$) and $p$ and $q$ are admissible.
		\ENSURE The polynomial $b_{pq}$. 
		\IF{$p=q$}
		\STATE $b_{pq}:=1$;
		\ELSE
		\STATE $p:=\bar p$
		\STATE $q:=q^p$;
		\STATE $T:=[p]\cup [\tau \ : \ \tau \text{ is $p$-admissible and }  \tau < q] \cup [q]$;
		\FOR{$\mu=1,\dots, \vert q-p \vert$}
		\FOR{$(\tau,\sigma)\in T\times T$ such that $\tau<\sigma$ and $\vert\sigma-\tau\vert=\mu$}
		\IF{$\sigma$ is $\tau$-admissible}
		\STATE $R_{\tau\sigma}:=a_{\tau\sigma} - \displaystyle{\sum_{\substack{\tau < \eta < \sigma\\   \hspace{2pt} \tau-adm.}}} g_{\tau\eta} b_{\eta q}$; 
		\STATE $g_{\tau\sigma}:=\tilde U_{\tau\sigma}(R_{\tau\sigma})$; 
		\STATE $b_{\tau\sigma}:=R_{\tau\sigma}-g_{\tau\sigma}$; 
		\ELSE 
		\STATE $a_{\tau\sigma}:=a_{\tau \sigma^\tau}$; $b_{\tau\sigma}:=b_{\tau \sigma^\tau}$;
		\ENDIF; 
		\ENDFOR
		\ENDFOR
		\ENDIF
		\STATE return $b_{pq}$
	\end{algorithmic}
\end{algorithm}

\begin{proposition}
	KaLu$\big(p, \mathcal{I}, \mathcal{J}, k, l, q \big)$ returns $b_{pq}$ and computes all the polynomials $b_{\tau \eta^{\tau}}$, with both $\tau$ and $\eta$ $p$-admissible and $\tau < \eta$, and $g_{\tau \eta}$ if $\eta$ is $\tau$-admissible, as well.
\end{proposition}

\begin{proof}
	This algorithm deals with a finite number of objects that are described by a finite number of data each. So the termination follows straightforwardly. For the correctness, we now analyse the command lines.
	
	We impose that $p$ is essential, that is we set $p:=\bar p$, and that all the Schubert varieties involved in the computation are represented referring to the essential pair $(\mathcal F_p,\mathcal I_p)$. Let $\omega$ be the length of $p$. 
	
	If $p \neq q$ the algorithm considers $\bar p$ and $q^p$ and computes the list $T$ of all the $\omega$-tuples $\tau< q$ that are $p$-admissible (line 7). Observe that if $\tau$ belongs to $T$ and another $\omega$-tuple $\sigma$ is $\tau$-admissible, then $\sigma$ is $p$-admissible too and hence belongs to $T$. The vice versa does not always hold.
	
	Then, for every $\mu$ between $1$ and $\vert q-p \vert$, the algorithm considers all the pairs $(\tau,\sigma)$ of elements in $T$ such that $\vert\sigma-\tau\vert=\mu$ (lines 8-9). If $\sigma$ is $\tau$-admissible, then $R_{\tau\sigma}$, $g_{\tau\sigma}$ and hence $b_{\tau\sigma}$ are computed by the formulas of Corollary~\ref{cor:Utilde}, being the explicit computation of $a_{\tau\sigma}$ possible thanks to Proposition~\ref{Prop_apq} (lines 11-13). Note that in order to apply \eqref{EqRpq} the algorithm must consider the values of $\mu$ in increasing order (line 8).
	
	If $\sigma$ is not $\tau$-admissible, the algorithm consider $\sigma^\tau$ in place of $\sigma$. In this case only the polynomials $a_{\tau\sigma}=a_{\tau \sigma^\tau}$ and $b_{\tau\sigma}=b_{\tau \sigma^\tau}$ are needed (line 15), where the equalities hold thanks to Remark~\ref{RemFibres} and Corollary~\ref{CorBpq}, as it has already been previously observed. As we have already pointed out, we have $\vert \sigma^\tau -\tau \vert < \vert \sigma -\tau \vert$, so $a_{\tau \sigma^\tau}$ and $b_{\tau \sigma^\tau}$ have already been computed.
	
	When $\mu$ reaches the value $\vert q-p \vert$, the pair $(p,q)$ is finally considered and then $b_{pq}$ can be computed. Indeed, at that moment all the necessary data to apply formula~\eqref{EqRpq} to this pair have been obtained and stored.
\end{proof}

\subsection{Polynomial identities and tests}\label{SubsecTests}

Here we shall discuss how we tested \textit{KaLu}. Roughly speaking, we are going to impose certain conditions on the maps $\pi_{p}$ and other resolutions $\xi_{p}$, which will be introduced presently. In these cases, the Kazhdan-Lusztig polynomials are immediate to determine and, consequently, we can compare them with the ones computed by \textit{KaLu}.

Let us begin by showing the relation we want to impose on $\pi_{p}$ and $\xi_{p}$. A resolution of singularities $\chi: X \rightarrow Y$, is said to be \textit{small} \cite[Definition 8.4.6]{KiWo2006} if and only if
\begin{equation*}
	\codim \{ y \in Y \hspace{0.165em} : \hspace{0.165em} \dim \chi^{-1}(y) \geq \alpha \} > 2 \alpha \quad \forall \alpha > 0.
\end{equation*}

The smallness of the resolution $\chi$ implies that $IC_{Y}^{\bullet} \cong R \chi_{*} \mathbb{Q}_{X} [\dim Y]$ \cite[Corollary, \textsection 6.2]{GoMa1983}. In particular, when $Y$ is a Schubert variety $\Delta_{p}$ for some vector $p$, the previous isomorphism gives
\begin{equation*}
	\mathcal{H}^{\alpha}(IC_{\Delta_{p}}^{\bullet})_{V} \cong H^{\alpha + \dim Y}(\chi^{-1}(V)) \quad \forall V \in \Delta_{p};
\end{equation*}
and, if $V \in \Delta_{q}^{0}$, with $\Delta_{q}$ a $\Delta_{p}$-variety, the Kazhdan-Lusztig polynomial corresponding to $\Delta_{p}$ and $\Delta_{q}$ coincides with the Poincaré polynomial of the fibre $\chi^{-1}(V)$ (see \cite[Theorem 4.4.7]{dCaMi2009}, \cite[Theorem 9.1.3]{BiLa2000}).

Let us introduce the maps $\xi_{p}$. Let $\Delta_{p}$ be an $\mathcal{S}$-variety and set
\begin{equation*}
	\mathcal{D}_{p} :=
	\begin{Bmatrix*}
		(V, U_{1}, \ldots, U_{\omega_{\bar{p}}})\\
		\in \mathbb{G}_{k}(\mathbb{C}^{l}) \times \mathbb{G}_{k + j_{1}^{\bar{p}} - i_{1}^{\bar{p}} - \bar{p}_{1}}(\mathbb{C}^{l}) \times \ldots \times \mathbb{G}_{k + j_{\omega_{\bar{p}}}^{\bar{p}} - i_{\omega_{\bar{p}}}^{\bar{p}} - \bar{p}_{\omega_{p}}}(\mathbb{C}^{l})\\
	\mbox{s.t.} \quad U_{1} \subset \ldots \subset U_{\omega_{\bar{p}}} \hspace{0.165em} \wedge \hspace{0.165em} U_{\alpha} \supseteq V + F_{j_{\alpha}^{\bar{p}}}, \quad \alpha = 1, \ldots, \omega_{\bar{p}}
	\end{Bmatrix*},
\end{equation*}
which is a \textit{smooth variety} (the proof is analogous to the one of Proposition~\ref{PropFlagSmooth}). The projection on the first factor
\begin{equation*}
	\xi_{p}: (V, U_{1}, \ldots, U_{\omega_{\bar{p}}}) \in \mathcal{D}_{p} \rightarrow V \in \Delta_{p}
\end{equation*}
is a \textit{resolution of singularities} (the proof is similar to the one of Corollary~\ref{CorResolPI}).

\begin{notation}
	From now on, we are going to change notations for legibility's sake. Whenever we deal with an $\mathcal{S}$-variety $\Delta_{p}$,
	\begin{itemize}
		\item	we shall assume that $p = \bar{p}$ and set $\nu := \omega_{p}$;
		\item	we are going to write $(\mathcal{F}_{p}, \mathcal{I}_{p}) = (F_{\zeta_{1}} \subset \ldots \subset F_{\zeta_{\nu}}, (\iota_{1}, \ldots, \iota_{\nu}))$ instead of $(F_{j_{1}^{p}} \subset \ldots \subset F_{j_{\nu}^{p}}, (i_{1}^{p} + p_{1}, \ldots, i_{\nu}^{p} + p_{\nu}))$.
	\end{itemize}
	Moreover, any $\Delta_{p}$-variety $\Delta_{q}$ is supposed to be described in terms of the flag $\mathcal{F}_{p}$ (in particular, $q$ is a $\nu$-tuple) and we set
	\begin{equation*}
		\varepsilon = (\varepsilon_{\alpha})_{\alpha = 1, \ldots, \nu} = (q_{\alpha} - p_{\alpha})_{\alpha = 1, \ldots, \nu}.
	\end{equation*}
\end{notation}

Let $V \in \Delta_{p}$. There is a $p$-admissible $q$ such that $V \in \Delta_{q}^{0}$, so the fibre of $\xi_{p}$ at $V$ is
\begin{align*}
	G_{pq} := \xi_{p}^{-1}(V) \cong
	\begin{Bmatrix*}
		(U_{1}, \ldots, U_{\nu})\\
		\in \mathbb{G}_{k + \zeta_{1} - \iota_{1}}(\mathbb{C}^{l}) \times \ldots \times \mathbb{G}_{k + \zeta_{\nu} - \iota_{\nu}}(\mathbb{C}^{l})\\
		\mbox{s.t.} \quad U_{1} \subset \ldots \subset U_{\nu} \hspace{0.165em} \wedge \hspace{0.165em} U_{\alpha} \supseteq V + F_{\zeta_{\alpha}}, \hspace{0.33em} \alpha = 1, \ldots, \nu
	\end{Bmatrix*}
	.
\end{align*}
Its Poincaré polynomial is
\begin{align*}
	H_{G_{pq}} &= H_{\mathbb{G}_{\varepsilon_{\nu}}(\mathbb{C}^{l - k - \zeta_{\nu} + \iota_{\nu} + \varepsilon_{\nu}})} \cdot \prod_{\alpha = 1}^{\nu - 1} H_{\mathbb{G}_{\varepsilon_{\alpha}}(\mathbb{C}^{\zeta_{\alpha + 1} - \iota_{\alpha + 1} - \zeta_{\alpha} + \iota_{\alpha} + \varepsilon_{\alpha}})}\\
	&= \frac{P_{l - k - \zeta_{\nu} + \iota_{\nu} + \varepsilon_{\nu}}}{P_{\varepsilon_{\nu}} P_{l - k - \zeta_{\nu} + \iota_{\nu}}} \cdot \prod_{\alpha = 1}^{\nu - 1} \frac{P_{\zeta_{\alpha + 1} - \iota_{\alpha + 1} - \zeta_{\alpha} + \iota_{\alpha} + \varepsilon_{\alpha}}}{P_{\varepsilon_{\alpha}} P_{\zeta_{\alpha + 1} - \iota_{\alpha + 1} - \zeta_{\alpha} + \iota_{\alpha}}}.
\end{align*}
and its dimension, which can be interpreted by means of Ferrer's diagrams as shown in Example~\ref{ExaFpqGpq}, is
\begin{equation}
	\dim G_{pq} = \varepsilon_{\nu} \lambda_{\iota_{\nu}}^{p} + \sum_{\alpha = 1}^{\nu - 1} \varepsilon_{\alpha}(\lambda_{\iota_{\alpha}}^{p} - \lambda_{\iota_{\alpha + 1}}^{p}).
\end{equation}

The reason why we introduced this new family of resolution of singularities is that it is easy to determine whether $\xi_{p}$ is small (the same holds for the maps $\pi_{p}$; see Lemmas~\ref{RemSmallness} and \ref{LemSmallnessZel}); therefore they provide more examples which can be used to test \textit{KaLu}. Indeed, as we recalled at the beginning of this subsection, when $\pi_{p}$ and $\xi_{p}$ are small, the Kazhdan-Lusztig polynomial corresponding to $\Delta_{p}$ and $\Delta_{q}$ coincides with the Poincaré polynomial of the fibres $F_{pq}$ and $G_{pq}$, respectively (see Corollary~\ref{CorPolId}). Such polynomials are explicit, hence they can be compared to the ones $b_{pq}$ computed by \textit{KaLu}.

The following result is nothing but a consequence of the definition of smallness.
\begin{remark}\label{RemSmallness}
	$\xi_{p}$ is small if and only if $m_{pq} = m_{p} - m_{q} > 2 \dim G_{pq}$ for all $\Delta_{p}$-varieties $\Delta_{q}$. Similarly, $\pi_{p}$ is small if and only if $m_{pq} = m_{p} - m_{q} > 2 \dim F_{pq}$ for all $\Delta_{p}$-varieties $\Delta_{q}$.
\end{remark}

Proposition~\ref{RemSmallness} allows us to check the smallness property by means of the Ferrer's diagrams, since all numbers $m_{\tau}$, $\dim F_{p \tau}$ and $\dim G_{p \tau}$ have a suitable representation. Nonetheless, it is not convenient to check either $m_{pq} > 2 \dim G_{pq}$ or $m_{pq} > 2 \dim F_{pq}$ \textit{for all} $p$-admissible $q$; yet, the combination of Proposition~\ref{RemSmallness} with the next lemma yields an easy-to-compute smallness characterization.

\textit{Put} $\iota_{0} = \lambda_{\iota_{\nu + 1}}^{p} = 0$.

\begin{lemma}\label{LemSmallnessZel}
	\textbf{\cite[p. 144]{Zel1983}}. For any $\Delta_{p}$-variety $\Delta_{q}$,
	\begin{equation}\label{EqLemSmallnessZel}
		m_{pq} = \sum_{\alpha = 1}^{\nu} \varepsilon_{\alpha}(\iota_{\alpha} - \iota_{\alpha - 1} + \lambda_{\iota_{\alpha}}^{p} - \lambda_{\iota_{\alpha + 1}}^{p}) + B(\varepsilon),
	\end{equation}
	where the form
	\begin{equation*}
		B(\varepsilon) = \varepsilon_{\nu}^{2} + \sum_{\alpha = 1}^{\nu - 1} \varepsilon_{\alpha}^{2} - \varepsilon_{\alpha} \varepsilon_{\alpha + 1}
	\end{equation*}
	is positive definite (see also Example~\ref{ExaFpqGpq}).
\end{lemma}

\begin{proposition}\label{PropSmallnessLighter}
	$\xi_{p}$ is small if and only if $\iota_{\alpha} - \iota_{\alpha - 1} \geq \lambda_{\iota_{\alpha}}^{p} - \lambda_{\iota_{\alpha + 1}}^{p}$ for any $\alpha \in \{ 1, \ldots, \nu \}$. Analogously, $\pi_{p}$ is small if and only if $\iota_{\alpha} - \iota_{\alpha - 1} \leq \lambda_{\iota_{\alpha}}^{p} - \lambda_{\iota_{\alpha + 1}}^{p}$ for any $\alpha \in \{ 1, \ldots, \nu \}$. In particular, both resolutions are small if equality holds.
\end{proposition}
\begin{proof}
	We are going to prove the statement for $\xi_{p}$ only.
	\begin{description}[leftmargin = 5mm]
		\item[$\Rightarrow$] Suppose that $\xi_{p}$ is small. By Proposition~\ref{RemSmallness}, for any $p$-admissible $q$,
		\begin{align*}
			0 &< m_{pq} - 2 \dim G_{pq}\\
			&= \varepsilon_{\nu}^{2} + \sum_{\alpha = 1}^{\nu - 1} \varepsilon_{\alpha}^{2} - \varepsilon_{\alpha} \varepsilon_{\alpha + 1} + \sum_{\alpha = 1}^{\nu}  \varepsilon_{\alpha}(\iota_{\alpha} - \iota_{\alpha - 1}) - \sum_{\alpha = 1}^{\nu} \varepsilon_{\alpha}(\lambda_{\iota_{\alpha}}^{p} - \lambda_{\iota_{\alpha + 1}}^{p}),
		\end{align*}
		where we used Lemma~\ref{LemSmallnessZel} for the equality. This relation holds, in particular, for all
		\begin{align*}
			q \in \{ v_{1}, \ldots, v_{\nu} \} = \{ p + (1, 0, \ldots, 0), \ldots, p + (0, \ldots, 0, 1) \}
		\end{align*}
		(they are all $p$-admissible because $\Delta_{p}$ has $\nu$ essential conditions). Hence, for any $\alpha \in \{ 1, \ldots, \nu \}$ such that $v_{\alpha}$ is $p$-admissible,
		\begin{equation*}
			0 < m_{pv_{\alpha}} - 2 \dim G_{pv_{\alpha}} = 1 + \iota_{\alpha} - \iota_{\alpha - 1} - (\lambda_{\iota_{\alpha}}^{p} - \lambda_{\iota_{\alpha + 1}}^{p});
		\end{equation*}
i.e.
		\begin{equation*}
			\iota_{\alpha} - \iota_{\alpha - 1} \geq \lambda_{\iota_{\alpha}}^{p} - \lambda_{\iota_{\alpha + 1}}^{p}.
		\end{equation*}
		\item[$\Leftarrow$] Assume that the inequality holds. For any $p$-admissible $q$,
		\begin{alignat*}{2}
			m_{pq} &= \sum_{\alpha = 1}^{\nu} \varepsilon_{\alpha}(\iota_{\alpha} - \iota_{\alpha - 1} + \lambda_{\iota_{\alpha}}^{p} - \lambda_{\iota_{\alpha + 1}}^{p}) + B(\varepsilon) && \quad \mbox{Formula } \eqref{EqLemSmallnessZel}\\
			&= \dim G_{pq} + \dim F_{pq} + B(\varepsilon) && \\
			&\geq 2 \dim G_{pq} + B(\varepsilon) && \quad \mbox{hypothesis}\\
			&> 2 \dim G_{pq}. && \quad \mbox{$B(\varepsilon)$ is positive definite}
		\end{alignat*}
		Proposition~\ref{RemSmallness} guarantees that $\xi_{p}$ is small. \qedhere
	\end{description}
\end{proof}

In terms of Ferrer's diagrams, $\xi_{p}$ is small if and only if the $\alpha$-th vertical line of the diagram of $\Delta_{p}$ is longer than its $\alpha$-th horizontal line for all $\alpha$, whereas $\pi_{p}$ is small if and only if the converse is true. If equality holds for all $\alpha$, both $\xi_{p}$ and $\pi_{p}$ are small.

\begin{corollary}[Polynomial identities]\label{CorPolId}
	If $\iota_{\alpha} - \iota_{\alpha - 1} \geq \lambda_{\iota_{\alpha}}^{p} - \lambda_{\iota_{\alpha + 1}}^{p}$ for any $\alpha \in \{ 1, \ldots, \nu \}$ (respectively, $\leq$), then the Poincaré polynomial $H_{G_{pq}}$ (respectively, $H_{F_{pq}}$) of the fibre equals $b_{pq}$ for any $p$-admissible $q$.
\end{corollary}

Let us conclude by listing the cases in which we tested \textit{KaLu}. We used the conditions of Proposition~\ref{PropSmallnessLighter} and, separately, the resolutions $\pi_{p}$ and $\xi_{p}$. We examined all Schubert varieties with at most 5 conditions and the restriction $l \leq 20$ and all Schubert varieties obtained imposing $i_{1} < \ldots < i_{10} < k \leq j_{1} < \ldots < j_{10} < l \leq 25$. The whole outputs of our tests are available at
\\
\url{http://wpage.unina.it/carmine.sessa2/KaLu/Tests\_Pi\_Small} and
\\
\url{http://wpage.unina.it/carmine.sessa2/KaLu/Tests\_Csi\_Small}.
\\
It is also worth stressing out that, for instance,
\begin{itemize}
	\item	$\pi_{p}$ is small as opposed to $\xi_{p}$ if\\
	$i_{1} = 1$, $i_{2} = 2$, $k = 3$, $j_{1} = 4$, $j_{2} = 6$ and $l = 9$;
	\item	$\xi_{p}$ is small as opposed to $\pi_{p}$ if\\
	$i_{1} = 1$, $i_{2} = 3$, $k = 4$, $j_{1} = 5$, $j_{2} = 8$ and $l = 10$;
	\item	both $\pi_{p}$ and $\xi_{p}$ are small if\\
	$i_{1} = 1$, $i_{2} = 2$, $k = 3$, $j_{1} = 4$, $j_{2} = 6$ and $l = 8$;
	\item	neither of $\pi_{p}$ and $\xi_{p}$ is small if\\
	$i_{1} = 2$, $i_{2} = 3$, $k = 4$, $j_{1} = 5$, $j_{2} = 7$ and $l = 10$.
\end{itemize}

\subsection{Relevant varieties}\label{SubsecRelevant}

In this conclusive paragraph, we show that, though $\pi_{0}$ is not small, there may be some $\mathcal{S}$-varieties which do not contribute to the decomposition of $R \pi_{0*} \mathbb{Q}_{\tilde{\mathcal{S}}}$ given by the combination of Theorem~\ref{ThMain} and decomposition theorem; that is, there are $\mathcal{S}$-varieties $\Delta_{p}$ such that $g_{0p} = 0$.

Let us restate the definition of relevant variety, given in the introduction, by means of the notations introduced so far.
\begin{definition}
	Given a Schubert variety $\mathcal{S}$, an $\mathcal{S}$-variety $\Delta_{q} \neq \mathcal{S}$ is said to be \textbf{$\pi_{0}$-relevant} if and only if $m_{0q} \leq 2 \dim F_{0q}$.
\end{definition}

We know that the smooth locus $\Delta_{q}^{0}$ of a Schubert variety $\Delta_{q} \subset \mathcal{S}$ cannot be a support of the decomposition unless $\Delta_{q}$ is an $\mathcal{S}$-variety. If $q$ is admissible and $m_{0q} > 2 \dim F_{0q}$, $\Delta_{q}$ does not provide any contribution in the decomposition. It might seem reasonable to expect that the converse occurs when $\Delta_{q}$ is $\pi_{0}$-relevant, yet, we are able to prove, by virtue of \textit{KaLu}, that this is not always the case. In Table~\ref{TabRelevant} there are some examples of $\pi_{0}$-relevant varieties whose contribution in the decomposition is null; i.e.~$g_{0q} = 0$. Richer lists are available in the ancillary files at \url{http://wpage.unina.it/carmine.sessa2/KaLu/Tests\_Relevant\_Varieties}.

\begin{table}
	\centering
	\begin{tabular}{c|c|c|c|c|c}
		$\omega$ & $I = [ i_{1}, \ldots, i_{\nu} ]$ & $k$ & $J = [ j_{1}, \ldots, j_{\nu} ]$ & $l$ & $q = [q_{1}, \ldots, q_{\nu}]$\\
		\hline
		2 & [3,  4] & 5 & [6,  8] & 11 & [2,  1]\\
		\hline
		\multirow{3}{*}{3} & \multirow{3}{*}{[3,  4,  5]} & \multirow{3}{*}{6} & \multirow{3}{*}{[7,  9,  11]} & \multirow{3}{*}{13} & [1,  0,  1]\\
		 & & & & & [1,  2,  1]\\
		 & & & & & [3,  2,  1]\\
		\hline
		\multirow{15}{*}{4} & \multirow{15}{*}{[3,  4,  5,  6]} & \multirow{15}{*}{7} & \multirow{15}{*}{[8,  10,  12,  14]} & \multirow{15}{*}{16} & [1,  0,  1,  0]\\
		 & & & & & [1,  0,  0,  1]\\
		 & & & & & [1,  1,  0,  1]\\
		 & & & & & [1,  0,  1,  1]\\
		 & & & & & [2,  1,  0,  1]\\
		 & & & & & [1,  2,  1,  0]\\
		 & & & & & [1,  2,  1,  1]\\
		 & & & & & [1,  1,  2,  1]\\
		 & & & & & [3,  2,  1,  0]\\
		 & & & & & [2,  1,  2,  1]\\
		 & & & & & [1,  2,  2,  1]\\
		 & & & & & [3,  2,  1,  1]\\
		 & & & & & [3,  2,  2,  1]\\
		 & & & & & [2,  3,  2,  1]\\
		 & & & & & [3,  3,  2,  1]
	\end{tabular}
	\caption{Here is, for $\omega = 2, 3, 4$, the first set of integers for which there are $\pi_{0}$-relevant varieties such that $g_{0q} = 0$.}
	\label{TabRelevant}
\end{table}

At the moment, we are not in position to explain the geometrical reason behind this phenomenon, although it is immediate to see that \textit{KaLu} gives $g_{0q} = 0$ because $g_{0q}$ is obtained by symmetrizing the polynomial $R_{0q}$ with respect to the degree $m_{0q} > \deg R_{0q}$. Furthermore, it would be interesting to understand if there exists a characterization of the $\pi_{0}$-relevant varieties which actually contribute to the decomposition theorem.

\section{Appendix: examples of Ferrer's diagrams}\label{SecExamples}

Here are a few examples of Ferrer's diagrams, each of which stresses out certain properties of Schubert varieties.

\begin{example}\label{ExaScV}
	Let $\mathcal{S}$ be the Schubert variety given by
	\begin{alignat*}{10}
		&i_{1} &&= 1, \qquad &&i_{2} &&= 2, \qquad &&i_{3} &&= 3, \qquad &&i_{4} &&= 4, \qquad &&k &&= 5,\\
		&j_{1} &&= 5, \qquad &&j_{2} &&= 7, \qquad &&j_{3} &&= 9, \qquad &&j_{4} &&= 11, \qquad &&l &&= 15.
	\end{alignat*}
	The sequence associated to $\mathcal{S}$ is $\lambda^{\mathcal{S}} = (6, 5, 4, 3, 0)$, which is shown in the picture below. If we take $p = (1, 1, 1, 1)$, the $\mathcal{S}$-variety $\Delta_{p}$, represented by the dashed diagram below, is associated to $\lambda^{p} = (7, 7, 6, 5, 4)$.
	
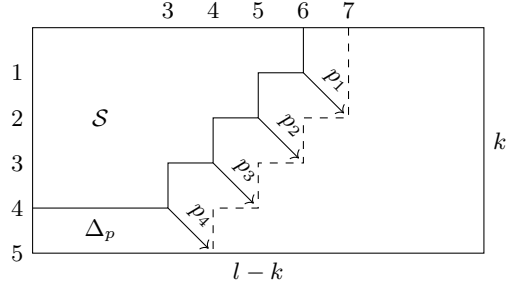
\begin{equation*}
		\begin{tikzpicture}[scale = 0.6]
			\draw (0, -1) node[anchor = east] {\small{1}};
			\draw (0, -2) node[anchor = east] {\small{2}};
			\draw (0, -3) node[anchor = east] {\small{3}};
			\draw (0, -4) node[anchor = east] {\small{4}};
			\draw (0, -5) node[anchor = east] {\small{5}};
			\draw (3, 0) node[anchor = south] {\small{3}};
			\draw (4, 0) node[anchor = south] {\small{4}};
			\draw (5, 0) node[anchor = south] {\small{5}};
			\draw (6, 0) node[anchor = south] {\small{6}};
			\draw (7, 0) node[anchor = south] {\small{7}};
			\draw (1.5, -2) node {\small{$\mathcal{S}$}};
			\draw (1.5, -4.5) node {\small{$\Delta_{p}$}};
			\draw (0, -5) rectangle (10, 0);
			\draw (6, 0) -- (6, -1) -- (5, -1) -- (5, -2) -- (4, -2) -- (4, -3) -- (3, -3) -- (3, -4) -- (0, -4);
			\draw[dashed] (7, 0) -- (7, -2) -- (6, -2) -- (6, -3) -- (5, -3) -- (5, -4) -- (4, -4) -- (4, -5);
			\draw[->] (6, -1) -- (6.9, -1.9) node[midway, sloped, above] {\small{$p_{1}$}};
			\draw[->] (5, -2) -- (5.9, -2.9) node[midway, sloped, above] {\small{$p_{2}$}};
			\draw[->] (4, -3) -- (4.9, -3.9) node[midway, sloped, above] {\small{$p_{3}$}};
			\draw[->] (3, -4) -- (3.9, -4.9) node[midway, sloped, above] {\small{$p_{4}$}};
			\draw (10, -2.5) node[anchor = west] {\small{$k$}};
			\draw (5, -5) node[anchor = north] {\small{$l - k$}};
		\end{tikzpicture}
	\end{equation*}
	Now, take $q = (1, 2, 1, 1)$. $j_{2} - i_{2} - q_{2} = j_{1} - i_{1} - q_{1} = 3$ and $i_{2} + q_{2} = i_{3} + q_{3} = 4$, thus the first and third conditions are unnecessary. In other words, $\mathcal{F}_{\bar{q}}: F_{j_{2}} \subset F_{j_{4}}$, $\mathcal{I}_{\bar{q}} = (4, 5)$ and $\bar{q} = (2, 1)$. Below, $\Delta_{q}$ is depicted by the dashed diagram.
	
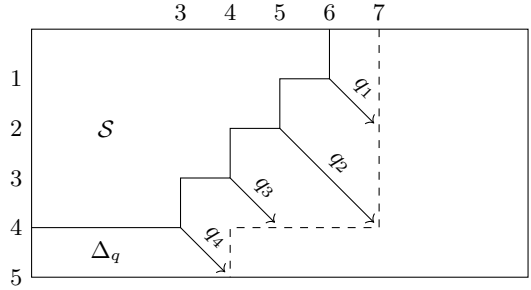
\begin{equation*}
		\begin{tikzpicture}[scale = 0.66]
			\draw (0, -1) node[anchor = east] {\small{1}};
			\draw (0, -2) node[anchor = east] {\small{2}};
			\draw (0, -3) node[anchor = east] {\small{3}};
			\draw (0, -4) node[anchor = east] {\small{4}};
			\draw (0, -5) node[anchor = east] {\small{5}};
			\draw (3, 0) node[anchor = south] {\small{3}};
			\draw (4, 0) node[anchor = south] {\small{4}};
			\draw (5, 0) node[anchor = south] {\small{5}};
			\draw (6, 0) node[anchor = south] {\small{6}};
			\draw (7, 0) node[anchor = south] {\small{7}};
			\draw (1.5, -2) node {\small{$\mathcal{S}$}};
			\draw (1.5, -4.5) node {\small{$\Delta_{q}$}};
			\draw (0, -5) rectangle (10, 0);
			\draw (6, 0) -- (6, -1) -- (5, -1) -- (5, -2) -- (4, -2) -- (4, -3) -- (3, -3) -- (3, -4) -- (0, -4);
			\draw[dashed] (7, 0) -- (7, -4) -- (4, -4) -- (4, -5); 
			\draw[->] (6, -1) -- (6.9, -1.9) node[midway, sloped, above] {\small{$q_{1}$}};
			\draw[->] (5, -2) -- (6.9, -3.9) node[midway, sloped, above] {\small{$q_{2}$}};
			\draw[->] (4, -3) -- (4.9, -3.9) node[midway, sloped, above] {\small{$q_{3}$}};
			\draw[->] (3, -4) -- (3.9, -4.9) node[midway, sloped, above] {\small{$q_{4}$}};
		\end{tikzpicture}
	\end{equation*}
	From the pictures we understand two important facts. First, \textit{the number of corners of an $\mathcal{S}$-variety $\Delta_{q'}$ equals the number of its essential conditions}. Secondly, \textit{when $\Delta_{q'}$ is associated to the flag $\mathcal{F}$, we can interpret the components of the $\omega$-tuple $q'$ as its distance from $\mathcal{S}$. In particular, the terms corresponding to the essential conditions measure the distance between the corners of $\mathcal{S}$ and $\Delta_{q}$}.
\end{example}

\begin{example}\label{Exakj}
	Here we exhibit and comment an example with $k > j_{1}$. Set	$i_{1} = 1$, $i_{2} = 5$, $k = 8$, $j_{1} = 6$, $j_{2} = 11$, $l = 15$ (diagram on the left) and $j_{1}' = 9$, $j_{2}' = 14$ and $l' = 18$ (diagram on the right) and let $\mathcal{S}$ and $\mathcal{S}'$ be the corresponding Schubert varieties.
	\\
	There are no $\mathcal{S}$-varieties whose top corner touches the bottom edge of the $k \times (l - k)$ rectangle: this happens because $i_{1} \leq \min \{ j_{1}, k \} = j_{1}$; that is, the first edge of the rectangle met by the corner is the right one (dashed diagram). Instead, there are some $\mathcal{S}$-varieties whose second corner intersects the bottom edge (dotted diagram).
	\\
	On the contrary, \textit{all corners of $\mathcal{S}'$ can be moved until they reach the bottom edge of the $k \times (l - k)$ rectangle because it is wide enough to let the top corner reach the bottom edge before the right one}. This also implies that some $\mathcal{S}'$-varieties are not $\mathcal{S}$-varieties, thus, when the decomposition theorem is applied to $\mathcal{S}$ and $\mathcal{S}'$, we see that \textit{$\mathcal{S}$ misses some of the supports of $\mathcal{S}'$}.
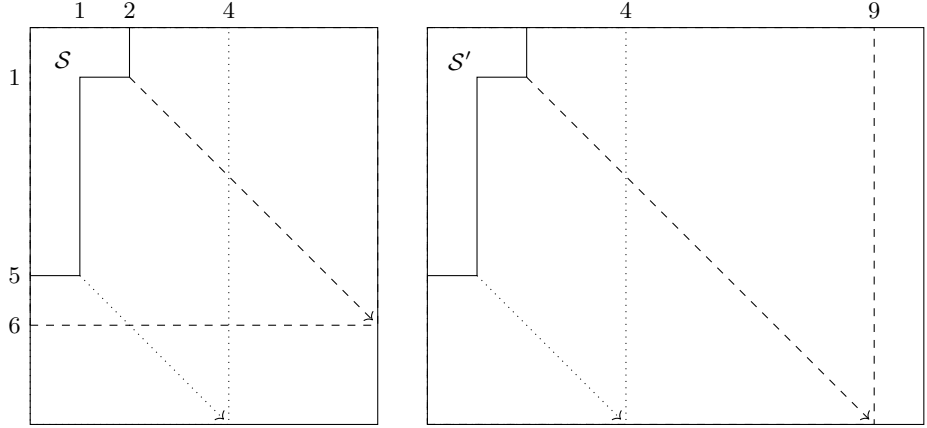
	\begin{equation*}
		\begin{tikzpicture}[scale = 0.66]
			\draw (0, -8) rectangle (7, 0);
			\draw (2, 0) -- (2, -1) -- (1, -1) -- (1, -5) -- (0, -5);
			\draw[dashed] (0, -6) rectangle (7, 0);
			\draw[->, dashed] (2, -1) -- (6.9, -5.9);
			\draw[dotted] (0, -8) rectangle (4, 0);
			\draw[->, dotted] (1, -5) -- (3.9, -7.9);
			\draw (1, 0) node[anchor = south] {\small{1}};
			\draw (2, 0) node[anchor = south] {\small{2}};
			\draw (4, 0) node[anchor = south] {\small{4}};
			\draw (0, -1) node[anchor = east] {\small{1}};
			\draw (0, -5) node[anchor = east] {\small{5}};
			\draw (0, -6) node[anchor = east] {\small{6}};
			\draw (0.66, -0.66) node {\small{$\mathcal{S}$}};
			\draw (8, -8) rectangle (18, 0);
			\draw (10, 0) -- (10, -1) -- (9, -1) -- (9, -5) -- (8, -5);
			\draw[dashed] (8, -8) rectangle (17, 0);
			\draw[->, dashed] (10, -1) -- (16.9, -7.9);
			\draw[dotted] (8, -8) rectangle (12, 0);
			\draw[->, dotted] (9, -5) -- (11.9, -7.9);
			\draw (12, 0) node[anchor = south] {\small{4}};
			\draw (17, 0) node[anchor = south] {\small{9}};
			\draw (8.66, -0.66) node {\small{$\mathcal{S}'$}};
		\end{tikzpicture}
	\end{equation*}
\end{example}

\begin{example}\label{ExaNotComp}
	Let $\mathcal{S}$ be the Schubert variety given by
	\begin{alignat*}{8}
		&i_{1} &&= 1, \qquad &&i_{2} &&= 3, \qquad &&i_{3} &&= 5, \qquad &&k &&= 7,\\
		&j_{1} &&= 8, \qquad &&j_{2} &&= 12, \qquad &&j_{3} &&= 17, \qquad &&l &&= 20.
	\end{alignat*}
	Consider the $\mathcal{S}$-varieties $\Delta_{p}$ (dashed) and $\Delta_{\bar{q}}$ (dotted), with $\bar{p} = 3$ and $\bar{q} = (3, 1)$, associated to their essential pairs $(\mathcal{F}_{\bar{p}}: F_{j_{2}}, \mathcal{I}_{\bar{p}} = 6)$ and $(\mathcal{F}_{\bar{q}}: F_{j_{1}} \subset F_{j_{3}}, \mathcal{I}_{\bar{q}} = (4, 6))$. Neither of their diagrams contains the other; therefore $\Delta_{p}$ and $\Delta_{q}$ are not comparable.
	
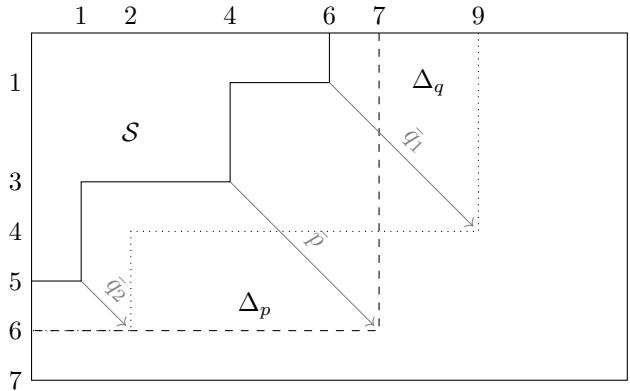
\begin{equation*}
		\begin{tikzpicture}[scale=0.66]
			\draw (0, -1) node[anchor = east] {1};
			\draw (0, -3) node[anchor = east] {3};
			\draw (0, -4) node[anchor = east] {4};
			\draw (0, -5) node[anchor = east] {5};
			\draw (0, -6) node[anchor = east] {6};
			\draw (0, -7) node[anchor = east] {7};
			\draw (1, 0) node[anchor = south] {1};
			\draw (2, 0) node[anchor = south] {2};
			\draw (4, 0) node[anchor = south] {4};
			\draw (6, 0) node[anchor = south] {6};
			\draw (7, 0) node[anchor = south] {7};
			\draw (9, 0) node[anchor = south] {9};
			\draw (2, -2) node {$\mathcal{S}$};
			\draw (8, -1) node {$\Delta_{q}$};
			\draw (4.5, -5.5) node {$\Delta_{p}$};
			\draw (0, -7) rectangle (12, 0);
			\draw (6, 0) -- (6, -1) -- (4, -1) -- (4, -3) -- (1, -3) -- (1, -5) -- (0, -5); 
			\draw[dashed] (7, 0) -- (7, -6) -- (0, -6); 
			\draw[dotted] (9, 0) -- (9, -4) -- (2, -4) -- (2, -5) -- (2, -6) -- (0, -6); 
			\draw[->, gray] (6, -1) -- (8.9, -3.9) node[midway, sloped, above] {$\bar{q}_{1}$};
			\draw[->, gray] (4, -3) -- (6.9, -5.9) node[midway, sloped, above] {$\bar{p}$};
			\draw[->, gray] (1, -5) -- (1.9, -5.9) node[midway, sloped, above] {$\bar{q}_{2}$};
		\end{tikzpicture}
	\end{equation*}
	\textit{The essential pair of an $\mathcal{S}$-variety is all we need to draw its Ferrer's diagram}. Anyway, if we wanted to describe $\Delta_{p}$ and $\Delta_{\bar{q}}$ by means of $\mathcal{F}$, then we would have $\mathcal{I}_{p} = (2, 6, 6)$, $\mathcal{I}_{q} = (4, 4, 6)$, $p = (1, 3, 1)$ and $q = (3, 1, 1)$. As you can see, $p_{1} < q_{1}$ and $p_{2} > q_{2}$, which confirms the fact that the studied varieties are not comparable.
\end{example}

\begin{example}\label{ExaComp}
	Let $\mathcal{S}$ be the Schubert variety given in Example~\ref{ExaNotComp} and take $p = (0, 2, 0)$. $\Delta_{p}$ (dashed) is a special Schubert variety and $\Delta_{q}$ (dotted), with $q = (2, 3, 1)$ is an $\mathcal{S}$-variety contained in $\Delta_{p}$ which is not a $\Delta_{p}$-variety (the essential flag of $\Delta_{q}$ is not a subflag of $\mathcal{F}_{\bar{p}}$).
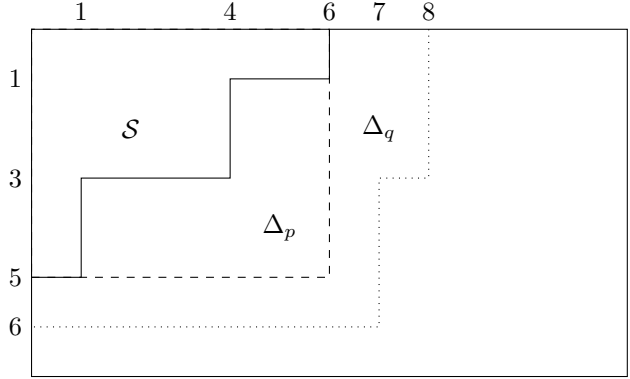
	\begin{equation*}
		\begin{tikzpicture}[scale=0.66]
			\draw (0, -1) node[anchor = east] {1};
			\draw (0, -3) node[anchor = east] {3};
			\draw (0, -5) node[anchor = east] {5};
			\draw (0, -6) node[anchor = east] {6};
			\draw (1, 0) node[anchor = south] {1};
			\draw (4, 0) node[anchor = south] {4};
			\draw (6, 0) node[anchor = south] {6};
			\draw (7, 0) node[anchor = south] {7};
			\draw (8, 0) node[anchor = south] {8};
			\draw (2, -2) node {$\mathcal{S}$};
			\draw (7, -2) node {$\Delta_{q}$};
			\draw (5, -4) node {$\Delta_{p}$};
			\draw (0, -7) rectangle (12, 0);
			\draw (6, 0) -- (6, -1) -- (4, -1) -- (4, -3) -- (1, -3) -- (1, -5) -- (0, -5); 
			\draw[dashed] (0, -5) rectangle (6, 0); 
			\draw[dotted] (8, 0) -- (8, -3) -- (7, -3) -- (7, -6) -- (0, -6); 
		\end{tikzpicture}
	\end{equation*}
\end{example}

\begin{example}\label{ExaAdapt}
	Let $\mathcal{S}$ be the Schubert variety given by
	\begin{alignat*}{8}
		&i_{1} &&= 2, \qquad &&i_{2} &&= 4, \qquad &&i_{3} &&= 6, \qquad &&k &&= 10,\\
		&j_{1} &&= 11, \qquad &&j_{2} &&= 14, \qquad &&j_{3} &&= 17, \qquad &&l &&= 22
	\end{alignat*}
	and put $p = (1, 2, 0)$, $q = (4, 3, 1)$, $q' = (2, 3, 2)$. $\Delta_{p}$, $\Delta_{q}$ and $\Delta_{q'}$ are represented by the dashed, dotted and dashed-dotted diagrams, respectively, and the grey circle highlights the common corner of $\Delta_{q}$ and $\Delta_{q'}$. $\Delta_{p}$ is a special Schubert variety, while $\Delta_{\bar{q}}$ and $\Delta_{\bar{q}'}$ have two conditions, but with respect to different flags. $\Delta_{q^{p}}$ and $\Delta_{q^{\prime p}}$ coincide with the special Schubert variety $\Delta_{(2, 3, 1)}$, whose only corner is represented by the grey circle. Trivially, $\Delta_{\bar{q}}, \Delta_{\bar{q}'} \neq \Delta_{q^{p}}$.
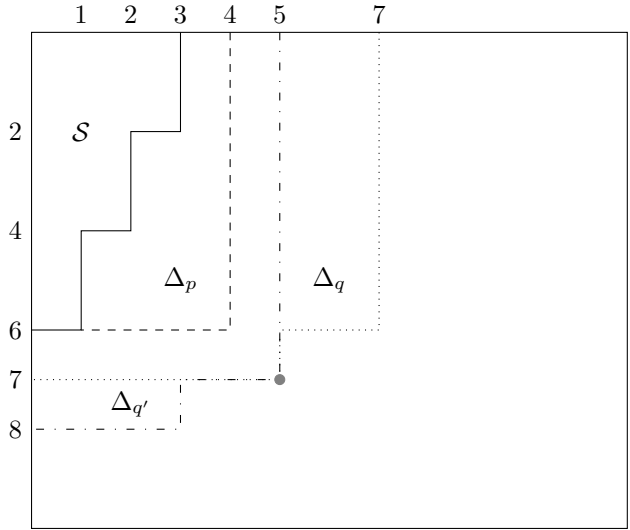
	\begin{equation*}
		\begin{tikzpicture}[scale = 0.66]
			\draw (0, -6) node[anchor = east] {6};
			\draw (0, -7) node[anchor = east] {7};
			\draw (0, -2) node[anchor = east] {2};
			\draw (0, -8) node[anchor = east] {8};
			\draw (0, -4) node[anchor = east] {4};
			\draw (3, 0) node[anchor = south] {3};
			\draw (1, 0) node[anchor = south] {1};
			\draw (2, 0) node[anchor = south] {2};
			\draw (4, 0) node[anchor = south] {4};
			\draw (5, 0) node[anchor = south] {5};
			\draw (7, 0) node[anchor = south] {7};
			\draw (1, -2) node {$\mathcal{S}$};
			\draw (3, -5) node {$\Delta_{p}$};
			\draw (6, -5) node {$\Delta_{q}$};
			\draw (2, -7.5) node {$\Delta_{q'}$};
			\draw (0, -10) rectangle (12, 0); 
			\draw (3, 0) -- (3, -2) -- (2, -2) -- (2, -4) -- (1, -4) -- (1, -6) -- (0, -6); 
			\draw[dashed] (4, 0) -- (4, -6) -- (1, -6); 
			\draw[dotted] (7, 0) -- (7, -6) -- (5, -6) -- (5, -7) -- (0, -7); 
			\draw[loosely dash dot] (5, 0) -- (5, -7) -- (3, -7) -- (3, -8) -- (0, -8); 
			\filldraw[gray] (5, -7) circle [radius = 1mm];
		\end{tikzpicture}
	\end{equation*}
\end{example}

\begin{example}\label{ExaFpqGpq}
	Let $\mathcal{S}$ be the Schubert variety in Example~\ref{ExaNotComp}.
	\begin{alignat*}{8}
		&i_{1} &&= 1, \qquad &&i_{2} &&= 3, \qquad &&i_{3} &&= 5, \qquad &&k &&= 7,\\
		&j_{1} &&= 8, \qquad &&j_{2} &&= 12, \qquad &&j_{3} &&= 17, \qquad &&l &&= 20.
	\end{alignat*}
	Take $\Delta_{p} = \mathcal{S}$ and $\Delta_{q}$ with $q = (1, 3, 2)$. $\codim_{\mathcal{S}} \Delta_{q}$ is easily seen to be given by Formula \eqref{EqLemSmallnessZel}. The grey rectangles in the pictures below represent	the fibre of $\xi_{p}$ (on the top left) at any point of $\Delta_{q}$; the one of $\pi_{p}$ (on the top right); the value $\varepsilon_{\alpha}^{2}$ (on the bottom left); the quantity $\varepsilon_{\alpha} \varepsilon_{\alpha + 1}$ (on the bottom right). Remember that the sum of $\varepsilon_{1}^{2}$, $\varepsilon_{2}^{2}$, $\varepsilon_{3}^{2}$, $\varepsilon_{1} \varepsilon_{2}$ and $\varepsilon_{2} \varepsilon_{3}$ is the definite positive form $B(\varepsilon)$ (see Lemma~\ref{LemSmallnessZel}).
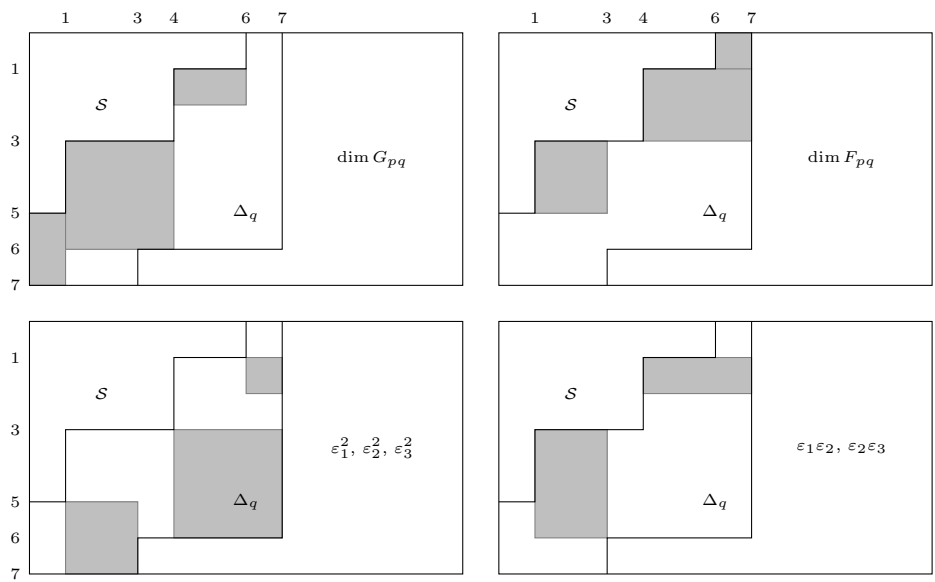
	\begin{equation*}
		\begin{tikzpicture}[scale=0.48]
			\draw (0, -1) node[anchor = east] {\tiny{1}};
			\draw (0, -3) node[anchor = east] {\tiny{3}};
			\draw (0, -5) node[anchor = east] {\tiny{5}};
			\draw (0, -6) node[anchor = east] {\tiny{6}};
			\draw (0, -7) node[anchor = east] {\tiny{7}};
			\draw (1, 0) node[anchor = south] {\tiny{1}};
			\draw (3, 0) node[anchor = south] {\tiny{3}};
			\draw (4, 0) node[anchor = south] {\tiny{4}};
			\draw (6, 0) node[anchor = south] {\tiny{6}};
			\draw (7, 0) node[anchor = south] {\tiny{7}};
			\draw (2, -2) node {\tiny{$\mathcal{S}$}};
			\draw (6, -5) node {\tiny{$\Delta_{q}$}};
			\draw (9.5, -3.5) node {\tiny{$\dim G_{pq}$}};
			\filldraw[fill = lightgray, draw = gray] (0, -7) rectangle (1, -5);
			\filldraw[fill = lightgray, draw = gray] (1, -6) rectangle (4, -3);
			\filldraw[fill = lightgray, draw = gray] (4, -2) rectangle (6, -1);
			\draw (0, -7) rectangle (12, 0);
			\draw (6, 0) -- (6, -1) -- (4, -1) -- (4, -3) -- (1, -3) -- (1, -5) -- (0, -5); 
			\draw (7, 0) -- (7, -6) -- (3, -6) -- (3, -7); 
			\draw (14, 0) node[anchor = south] {\tiny{1}};
			\draw (16, 0) node[anchor = south] {\tiny{3}};
			\draw (17, 0) node[anchor = south] {\tiny{4}};
			\draw (19, 0) node[anchor = south] {\tiny{6}};
			\draw (20, 0) node[anchor = south] {\tiny{7}};
			\draw (15, -2) node {\tiny{$\mathcal{S}$}};
			\draw (19, -5) node {\tiny{$\Delta_{q}$}};
			\draw (22.5, -3.5) node {\tiny{$\dim F_{pq}$}};
			\filldraw[fill = lightgray, draw = gray] (19, -1) rectangle (20, 0);
			\filldraw[fill = lightgray, draw = gray] (17, -3) rectangle (20, -1);
			\filldraw[fill = lightgray, draw = gray] (14, -5) rectangle (16, -3);
			\draw (13, -7) rectangle (25, 0);
			\draw (19, 0) -- (19, -1) -- (17, -1) -- (17, -3) -- (14, -3) -- (14, -5) -- (13, -5); 
			\draw (20, 0) -- (20, -6) -- (16, -6) -- (16, -7); 
			\draw (0, -9) node[anchor = east] {\tiny{1}};
			\draw (0, -11) node[anchor = east] {\tiny{3}};
			\draw (0, -13) node[anchor = east] {\tiny{5}};
			\draw (0, -14) node[anchor = east] {\tiny{6}};
			\draw (0, -15) node[anchor = east] {\tiny{7}};
			\draw (2, -10) node {\tiny{$\mathcal{S}$}};
			\draw (9.5, -11.5) node {\tiny{$\varepsilon_{1}^{2}$, $\varepsilon_{2}^{2}$, $\varepsilon_{3}^{2}$}};
			\filldraw[fill = lightgray, draw = gray] (6, -10) rectangle (7, -9);
			\filldraw[fill = lightgray, draw = gray] (4, -14) rectangle (7, -11);
			\filldraw[fill = lightgray, draw = gray] (1, -15) rectangle (3, -13);
			\draw (6, -13) node {\tiny{$\Delta_{q}$}};
			\draw (0, -15) rectangle (12, -8);
			\draw (6, -8) -- (6, -9) -- (4, -9) -- (4, -11) -- (1, -11) -- (1, -13) -- (0, -13); 
			\draw (7, -8) -- (7, -14) -- (3, -14) -- (3, -15); 
			\draw (15, -10) node {\tiny{$\mathcal{S}$}};
			\draw (19, -13) node {\tiny{$\Delta_{q}$}};
			\draw (22.5, -11.5) node {\tiny{$\varepsilon_{1} \varepsilon_{2}$, $\varepsilon_{2} \varepsilon_{3}$}};
			\filldraw[fill = lightgray, draw = gray] (14, -14) rectangle (16, -11);
			\filldraw[fill = lightgray, draw = gray] (17, -10) rectangle (20, -9);
			\draw (13, -15) rectangle (25, -8);
			\draw (19, -8) -- (19, -9) -- (17, -9) -- (17, -11) -- (14, -11) -- (14, -13) -- (13, -13); 
			\draw (20, -8) -- (20, -14) -- (16, -14) -- (16, -15); 
		\end{tikzpicture}
	\end{equation*}
\end{example}

\bibliographystyle{amsplain}
\bibliography{CioffiFrancoSessa}

\providecommand{\bysame}{\leavevmode\hbox to3em{\hrulefill}\thinspace}
\providecommand{\MR}{\relax\ifhmode\unskip\space\fi MR }
\providecommand{\MRhref}[2]{%
  \href{http://www.ams.org/mathscinet-getitem?mr=#1}{#2}
}
\providecommand{\href}[2]{#2}
\begin{thebibliography}{10}

\bibitem{CoCoA5}
J.~Abbott, A.~M. Bigatti, and L.~Robbiano, \emph{{CoCoA}: a system for doing
  {C}omputations in {C}ommutative {A}lgebra}, Available at
  \texttt{http://cocoa.dima.unige.it}.

\bibitem{BeBeDe1982}
A.~A. Be\u{\i}linson, J.~Bernstein, and P.~Deligne, \emph{Faisceaux pervers},
  Analysis and topology on singular spaces, {I} ({L}uminy, 1981),
  Ast\'{e}risque, vol. 100, Soc. Math. France, Paris, 1982, pp.~5--171.
  \MR{751966}

\bibitem{BiLa2000}
Sara Billey and V.~Lakshmibai, \emph{Singular loci of {S}chubert varieties},
  Progress in Mathematics, vol. 182, Birkh\"{a}user Boston, Inc., Boston, MA,
  2000. \MR{1782635}

\bibitem{ChGoMa1982}
Jeff Cheeger, Mark Goresky, and Robert MacPherson, \emph{{$L^{2}$}-cohomology
  and intersection homology of singular algebraic varieties}, Seminar on
  {D}ifferential {G}eometry, Ann. of Math. Stud., vol. 102, Princeton Univ.
  Press, Princeton, N.J., 1982, pp.~303--340. \MR{645745}

\bibitem{CiFrSe2021}
Francesca Cioffi, Davide Franco, and Sessa Carmine, \emph{Polynomial identities
  related to special {S}chubert varieties}, AAECC (2021), 1--21.

\bibitem{dCa2013}
Mark Andrea~A. de~Cataldo, \emph{Hodge-theoretic splitting mechanisms for
  projective maps}, J. Singul. \textbf{7} (2013), 134--156, With an appendix
  containing a letter from P. Deligne. \MR{3077721}

\bibitem{dCaMi2005}
Mark Andrea~A. de~Cataldo and Luca Migliorini, \emph{The {H}odge theory of
  algebraic maps}, Ann. Sci. \'{E}cole Norm. Sup. (4) \textbf{38} (2005),
  no.~5, 693--750. \MR{2195257}

\bibitem{dCaMi2009}
\bysame, \emph{The decomposition theorem, perverse sheaves and the topology of
  algebraic maps}, Bull. Amer. Math. Soc. (N.S.) \textbf{46} (2009), no.~4,
  535--633. \MR{2525735}

\bibitem{dGeFr2014}
Vincenzo Di~Gennaro and Davide Franco, \emph{Noether-{L}efschetz theory with
  base locus}, Rend. Circ. Mat. Palermo (2) \textbf{63} (2014), no.~2,
  257--276. \MR{3232654}

\bibitem{dGeFr2017OnTheExistence}
\bysame, \emph{On the existence of a {G}ysin morphism for the blow-up of an
  ordinary singularity}, Ann. Univ. Ferrara Sez. VII Sci. Mat. \textbf{63}
  (2017), no.~1, 75--86. \MR{3651640}

\bibitem{dGeFr2017OnTheTopology}
\bysame, \emph{On the topology of a resolution of isolated singularities}, J.
  Singul. \textbf{16} (2017), 195--211. \MR{3725396}

\bibitem{dGeFr2019}
\bysame, \emph{On a resolution of singularities with two strata}, Results Math.
  \textbf{74} (2019), no.~3, Paper No. 115, 22. \MR{3953478}

\bibitem{dGeFr2020}
\bysame, \emph{On the topology of a resolution of isolated singularities,
  {II}}, J. Singul. \textbf{20} (2020), 95--102. \MR{4083726}

\bibitem{Dim2004}
Alexandru Dimca, \emph{Sheaves in topology}, Universitext, Springer-Verlag,
  Berlin, 2004. \MR{2050072}

\bibitem{Fra2020}
Davide Franco, \emph{Explicit decomposition theorem for special {S}chubert
  varieties}, Forum Math. \textbf{32} (2020), no.~2, 447--470. \MR{4069946}

\bibitem{GoMa1980}
Mark Goresky and Robert MacPherson, \emph{Intersection homology theory},
  Topology \textbf{19} (1980), no.~2, 135--162. \MR{572580}

\bibitem{GoMa1983}
\bysame, \emph{Intersection homology. {II}}, Invent. Math. \textbf{72} (1983),
  no.~1, 77--129. \MR{696691}

\bibitem{GrHa1994}
Phillip Griffiths and Joseph Harris, \emph{Principles of algebraic geometry},
  Wiley Classics Library, John Wiley \& Sons, Inc., New York, 1994, Reprint of
  the 1978 original. \MR{1288523}

\bibitem{Ive1986}
Birger Iversen, \emph{Cohomology of sheaves}, Universitext, Springer-Verlag,
  Berlin, 1986. \MR{842190}

\bibitem{KiWo2006}
Frances Kirwan and Jonathan Woolf, \emph{An introduction to intersection
  homology theory}, second ed., Chapman \& Hall/CRC, Boca Raton, FL, 2006.
  \MR{2207421}

\bibitem{Man2001}
Laurent Manivel, \emph{Symmetric functions, {S}chubert polynomials and
  degeneracy loci}, SMF/AMS Texts and Monographs, vol.~6, American Mathematical
  Society, Providence, RI; Soci\'{e}t\'{e} Math\'{e}matique de France, Paris,
  2001, Translated from the 1998 French original by John R. Swallow, Cours
  Sp\'{e}cialis\'{e}s [Specialized Courses], 3. \MR{1852463}

\bibitem{Nav1985}
V.~Navarro~Aznar, \emph{Sur la th\'{e}orie de {H}odge des vari\'{e}t\'{e}s
  alg\'{e}briques \`a singularit\'{e}s isol\'{e}es}, no. 130, 1985,
  Differential systems and singularities (Luminy, 1983), pp.~272--307.
  \MR{804059}

\bibitem{Sai1986}
Morihiko Saito, \emph{Mixed {H}odge modules}, Proc. Japan Acad. Ser. A Math.
  Sci. \textbf{62} (1986), no.~9, 360--363. \MR{888148}

\bibitem{Wil2017}
Geordie Williamson, \emph{The {H}odge theory of the decomposition theorem}, no.
  390, 2017, S\'{e}minaire Bourbaki. Vol. 2015/2016. Expos\'{e}s 1104--1119,
  pp.~Exp. No. 1115, 335--367. \MR{3666031}

\bibitem{Zel1983}
A.~V. Zelevinski\u{\i}, \emph{Small resolutions of singularities of {S}chubert
  varieties}, Funktsional. Anal. i Prilozhen. \textbf{17} (1983), no.~2,
  75--77. \MR{705051}

\end{thebibliography}

\end{document}